\documentclass[a4paper,10pt]{article}
\usepackage[utf8]{inputenc}

\usepackage[backend=biber, style=alphabetic]{biblatex}
\usepackage{amsmath,amssymb,amsthm}
\usepackage{amscd}
\usepackage{mathrsfs}
\usepackage[usenames,dvipsnames]{color}
\usepackage{pstricks} 
\usepackage{hyperref}
\usepackage{graphicx}
\usepackage{verbatim}
\usepackage{geometry}
\usepackage{extarrows}
\usepackage{tabularx}
\usepackage{slashed}
\usepackage{xparse}
\usepackage{enumitem}

\usepackage{xspace}
\usepackage{allrunes}

\usepackage[all,cmtip]{xy}

\def\shuffle{{\sqcup\mathchoice{\mkern-3mu}{\mkern-3mu}{\mkern-3.2mu}{\mkern-3.8mu}\sqcup}} 
\def\stuffle{{\sqcup\mathchoice{\mkern-12.5mu}{\mkern-12.5mu}{\mkern-8.7mu}{\mkern-8.5mu} - \mathchoice{\mkern-12.5mu}{\mkern-12.5mu}{\mkern-8.5mu}{\mkern-8.7mu}\sqcup}}

\usepackage{tikz}
\tikzstyle{every picture}=[level distance = 8mm, baseline=-0.5ex]
\tikzstyle{prop}=[shape=circle,minimum size=6mm, draw=black!80, fill=green!30]

\usetikzlibrary{arrows,matrix}
\usetikzlibrary{decorations.pathmorphing}

\renewcommand{\s}{\mathfrak{s}}

\newcommand{\C}{\mathbb{C}}
\renewcommand{\R}{\mathbb{R}}
\renewcommand{\Q}{\mathbb{Q}}
\newcommand{\Z}{\mathbb{Z}}
\renewcommand{\N}{\mathbb{N}}

\newcommand{\yew}{\textarc{I}}

\newtheorem{thm}{Theorem}[section]
\newtheorem*{thm*}{Theorem}
\newtheorem{lem}[thm]{Lemma}
\newtheorem{coro}[thm]{Corollary}

\newtheorem{prop}[thm]{Proposition}

\theoremstyle{definition}
\newtheorem{defn}[thm]{Definition}
\newtheorem{rk}[thm]{Remark}
\newtheorem{ex}[thm]{Example}

\newtheorem{innercustomgeneric}{\customgenericname}
\providecommand{\customgenericname}{}
\newcommand{\newcustomtheorem}[2]{%
  \newenvironment{#1}[1]
  {%
   \renewcommand\customgenericname{#2}%
   \renewcommand\theinnercustomgeneric{##1}%
   \innercustomgeneric
  }
  {\endinnercustomgeneric}
}

\newcustomtheorem{customthm}{Theorem}

\newcommand {\fraks}{{\mathfrak {s}}}

\newcommand {\calc}{{\mathcal {C}}}
\newcommand {\calct}{{\mathcal {CT}}}

\newcommand {\calf}{{\mathcal{F}}}

\newcommand {\calt}{{\mathcal {T}}}

\newcommand {\calw}{{\mathcal {W}}}

\newcommand{\tdun}[1]{\begin{picture}(10,5)(-2,-1)
\put(0,0){\circle*{2}}
\put(2,-2){\tiny #1}
\end{picture}}

\newcommand{\tddeux}[2]{\begin{picture}(12,5)(0,-1)
\put(3,0){\circle*{2}}
\put(3,0){\line(0,1){5}}
\put(3,5){\circle*{2}}
\put(4,-2){\tiny #1}
\put(4,4){\tiny #2}
\end{picture}}

\newcommand{\tdtrois}[3]{\begin{picture}(20,20)(0,-1)
\put(3,0){\circle*{2}}
\put(3,0){\line(0,1){5}}
\put(3,5){\circle*{2}}
\put(3,5){\line(0,1){5}}
\put(3,10){\circle*{2}}
\put(4,-2){\tiny #1}
\put(4,4){\tiny #2}
\put(4,9){\tiny #3}
\end{picture}}

\newcommand{\tdtroisun}[3]{\begin{picture}(20,12)(-5,-1)
\put(3,0){\circle*{2}}
\put(-0.65,0){$\vee$}
\put(6,7){\circle*{2}}
\put(0,7){\circle*{2}}
\put(5,-2){\tiny #1}
\put(6,5){\tiny #2}
\put(-5,8){\tiny #3}
\end{picture}}

\newcommand{\tdquatre}[4]{\begin{picture}(20,20)(0,-1)
\put(3,-5){\circle*{2}}
\put(3,-5){\line(0,1){5}}
\put(3,0){\circle*{2}}
\put(3,0){\line(0,1){5}}
\put(3,5){\circle*{2}}
\put(3,5){\line(0,1){5}}
\put(3,10){\circle*{2}}
\put(4,-7){\tiny #1}
\put(4,-1){\tiny #2}
\put(4,4){\tiny #3}
\put(4,9){\tiny #4}
\end{picture}}

\newcommand{\tdquatredeux}[4]{\begin{picture}(20,20)(-5,-1)
\put(3,0){\circle*{2}}
\put(-.65,0){$\vee$}
\put(6,7){\circle*{2}}
\put(0,7){\circle*{2}}
\put(0,14){\circle*{2}}
\put(0,7){\line(0,1){7}}
\put(5,-2){\tiny #1}
\put(9,5){\tiny #2}
\put(-5,5){\tiny #3}
\put(-5,12){\tiny #4}
\end{picture}}

\newcommand{\tcinq}[5]{\begin{picture}(15,26)(-5,-1)
\put(3,0){\circle*{2}} 
\put(-0.65,0){$\vee$}
\put(0,5){\circle*{2}} 
\put(0,5){\line(0,1){7}}
\put(0,12){\circle*{2}} 
\put(6,5){\circle*{2}} 
\put(6,5){\line(0,1){7}}
\put(6,12){\circle*{2}} 

\put(5,0){\tiny #1}
\put(-5,5){\tiny #2}
\put(-5,10){\tiny #3}

\put(7,5){\tiny #4}
\put(7,12){\tiny #5}
\end{picture}}

\newcommand{\tcinqonze}[5]{\begin{picture}(15,26)(-5,-1)
\put(3,5){\circle*{2}} 
\put(-0.65,5){$\vee$}
\put(6,12){\circle*{2}} 
\put(0,12){\circle*{2}}
\put(3,0){\circle*{2}} 
\put(3,0){\line(0,1){5}}
\put(0,12){\line(0,1){7}} 
\put(0,19){\circle*{2}} 

\put(5,0){\tiny #1}
\put(5,5){\tiny #2}
\put(-5,12){\tiny #3}
\put(-2,21){\tiny #4}
\put(7,11){\tiny #5}
\end{picture}}

\begin{document}

\title{Generalisations of multiple zeta values to rooted forests}
\author{Pierre~J.~Clavier${}^{1}$, Dorian Perrot${}^{2,3}$\\
~\\
\normalsize \it $^1$  Department of Mathematics, IRIMAS, \\
\normalsize \it Université de Haute Alsace.\\
\normalsize \it $^2$ ENS Rennes.\\
\normalsize \it $^3$ Université de Rennes.\\
~\\
\normalsize email: pierre.clavier@uha.fr}

\date{}

\maketitle

\begin{abstract} 
 We show that any convergent (shuffle) arborified zeta value admits a series representation. This justifies the introduction of a new generalisation to rooted forests of multiple zeta values, and we study  its algebraic properties. As a consequence of the series representation, we derive elementary proofs of some results of Bradley and Zhou for Mordell-Tornheim zeta values and give explicit formulas. The series representation for shuffle arborified zeta values also implies that they are conical zeta values. We characterise which conical zeta values are arborified zeta values and evaluate them as sums of multiple zeta values with rational coefficients.
\end{abstract}

\tableofcontents

\section*{Introduction}

\addcontentsline{toc}{section}{Introduction}

\subsection*{Multiple zeta values}

\addcontentsline{toc}{subsection}{Multiple zeta values and objectives}

Multiple zeta values (MZVs) appeared for the first time in the work of Euler \cite{Eu1796}. After having appeared in various disguise in the work of many authors for two centuries, they eventually appeared in full generality as an application to Ecalle's work on mould calculus \cite{Ecalle}. A systematic studies was then undertaken by Hoffman \cite{Ho92} and Zagier \cite{Za91}. Many important conjectures about the relations obeyed by MZVs and their transcendentality are still an important field of research. MZVs and their generalisations have also appeared in multiple domains of Physics, for example when computing amplitudes in some quantum field or string theories, see for example \cite{To17}. 

For the interested reader, let us point out that there exists many classical introductions to MZVs, see for example \cite{Wa11}. We do not aim at writing another introduction to MZVs here, but simply to present some properties of MZVs in order for us to latter state the objectives of this paper.

A word written in the alphabet $\{x,y\}$ is called {\bf convergent} if it starts with $x$ and ends with $y$. For every convergent word $w\in\calw_{\{x,y\}}^{\rm conv}$ the 
{\bf shuffle multiple zeta value (MZV)} associated to $w$ is given by the image of the map
\begin{align} 
 \zeta_\shuffle :~ & \calw_{\{x,y\}}^{\rm conv}\subseteq\calw_{\{x,y\}} \longrightarrow \R \nonumber \\
		  & (\epsilon_1\cdots \epsilon_k) \mapsto \int_{1\geq t_1\geq \cdots \geq t_k\geq0}\prod_{i=1}^k \omega_{\epsilon_i}(t_i) \label{eq:integral_shuffle_MZV}
\end{align}
(with $\omega_x(t)=dt/t$, $\omega_y(t)=dt/(1-t)$) evaluated at $w$. On the other hand, a word written in the alphabet $\N^*$ is also called {\bf convergent} if it does not start with $1$. For every such convergent word $w\in\calw_{\N^*}^{\rm conv}$ the 
{\bf stuffle multiple zeta value} associated to $w$ is given by the image of the map
\begin{align} 
 \zeta_\stuffle :~ & \calw_{\N^*}^{\rm conv}\subseteq\calw_{\N^*} \longrightarrow \R \nonumber\\
		  & (p_1\cdots p_k) \mapsto \sum_{n_1>\cdots>n_k>0}\frac{1}{n_1^{p_1}\cdots n_k^{p_k}} \label{eq:sum_stuffle_MZV}
\end{align}
evaluated at $w$.

The maps $\zeta_\shuffle$ and $\zeta_\stuffle$ are algebra morphisms for the {\bf shuffle} of words, written $\shuffle$, and the {\bf stuffle} of words (also called quasi-shuffle and sticky shuffle in the literature), written $\stuffle$ respectively. They are also related by {\bf Kontsevitch's relation} via a binarisation map $\s:\calw_{\N^*}\longrightarrow\calw_{\{x,y\}}$. These various combinatorial objects will be rigorously introduced in \S \ref{ref:subsec:MZVs}.


The main objective of these paper is to extend these three properties of MZVs (namely shuffle algebra morphism, stuffle algebra morphism and Kontsevitch's relation) from words to rooted forests for a generalistion of MZVs.

\subsection*{Generalisations of multiple zeta values}

\addcontentsline{toc}{subsection}{Generalisations of multiple zeta values}

There exists multiple generalisations of MZVs, for example Euler sums, Hurwitz multiple zetas, (multiple) polylogarithms, Shintani zetas, Witten zetas... In this paper, we start by studying one of these generalisations, namely {\bf arborified zeta values} (AZVs) and apply our result to other generalisations, in particular {\bf conical zeta values} (CZVs).

AZVs appeared in the work of Ecalle \cite{Ecalle} and much later in the work of Yamamoto \cite{Ya14}. Their extensive study started in \cite{Ma13} and was completed in \cite{Cl20}. The renormalisation of their divergent counterparts was performed in \cite{CGPZ18}. In particular, it was shown in \cite{Cl20} that AZVs are linear combinations of MZVs with rational coefficients.
In the same paper, two generalisations to rooted forests of the shuffle and stuffle products were built for which AZVs are algebra morphisms. However, it was also shown in the same paper that, while AZVs exists as iterated sums and iterated integrals like MZVs, these two objects are not related by the most natural generalisation of the binarisation map $\fraks$. In other words: the natural generalisation of Kontsevitch's relation does not hold for AZVs.


In this paper, we also study Mordell-Tornheim zetas \cite{To50,Mo58,Ho92} and {\bf conical zeta values} (CZVs). The latter have been defined in \cite{GPZ13} and their divergent counterparts have been renormalised in \cite{GPZ17}. An important open question which is given a partial solution in this paper is to characterise when CZVs are linear combinations of MZVs with rational coefficients. Notice that this question may be of interest in physics since some CZVs have been shown to appear in the perturbative expansion of amplitudes of some string theories, see for example \cite{Ze16,Ze17}.


\subsection*{Main results and plan of the paper}

\addcontentsline{toc}{subsection}{Main results and plan of the paper}

Section \ref{section:two} starts with a presentation of some notions of combinatorics of words that are useful for MZVs, and state the properties of MZVs that we will generalise to rooted forests. Then we recall some definitions from graphs theory and combinatorics of rooted forests and state results concerning (shuffle) arborified zeta values we build upon. The section ends with a statement and a proof of Theorem \ref{thm:integral_sum}, which gives a new series representation for arborified zeta values. This result is one of the most important of this paper.

This main theorem justifies the definition of {\bf tree zeta values} (Definition \ref{defn:tree_zeta_values}). The study of tree zeta values is the main topic of Section \ref{section:two_bis}. Theorem \ref{thm:tree_zeta_MZVs} is then a direct consequence of Theorem \ref{thm:integral_sum} together with the aforementioned results of \cite{Cl20} on arborified zeta values. We then introduce the yew product (Definition \ref{defn:yew}) and give an explicit description of this product (Theorem \ref{thm:yew_formula}). We then show that TZVs form an algebra morphism for the yew product (Theorem \ref{thm:yew_TZVs}) and relate this product with various combinatorial objects (Theorem \ref{thm:comm_fl_fl_yew}); which allows us to directly relate TZVs and stuffle MZVs (Corollary \ref{coro:TZVs_stuffle_MZVs}). We end Section \ref{section:two_bis} with a new generalisation to rooted forests of the stuffle product (Definition \ref{defn:new_stuffle}) for which TZVs form an algebra morphism (Theorem \ref{thm:TZVs_alg_morph}).

In Section \ref{section:three} we present a second application of Theorem \ref{thm:integral_sum} which concerns Mordell-Tornheim zeta values. We give in Proposition \ref{prop:MT_s1_0} a formula for a special family of Mordell-Tornheim zeta values and give a new proof of a classical Theorem of Bradley and Zhou (Theorem \ref{thm:MT_tree}). Along the way we prove a decomposition formula for Mordell-Tornheim zeta values (Equation \eqref{eq:MT_decom}) which implies a formula for generic Mordell-Tornheim zeta values (Equation \eqref{eq:expression_MT}).

Our main application of Theorem \ref{thm:integral_sum} concerns conical zeta values and can be found in Section \ref{section:four}. After recalling some classical definitions of the theory of cones and CZVs, we show (Proposition \ref{prop:forest_cones}) that shuffle arborified zeta values are CZVs. This gives a formula for CZVs that can be obtained from rooted forests (Theorem \ref{thm:tree_CZVs_MZVs}). Such cones are characterised in Proposition \ref{prop:cone_to_tree}. Putting together the results of Sections \ref{section:two} and \ref{section:four} we obtain the second main result of the paper, Theorem \ref{thm:main_result}, which gives a sufficient condition on cones for the associated CZVs to a linear combination of MZVs with rational coefficients.

The paper ends with Section \ref{section:five} where evaluations of CZVs are performed in details using the methods developed in Sections \ref{section:two_bis} and \ref{section:four}.

\subsection*{Notations}

Through the paper, we use $[n]:=\{0,1,\cdots,n\}$. We also use $\N:=\Z_{\geq0}$, $\N^*:=\Z_{\geq1}$, $\R_+:=\R_{\geq0}$ and $\R^*_+:=\R_{>0}$.

\section{Series representation of arborified zetas} \label{section:two}

\subsection{Multiple zeta values} \label{ref:subsec:MZVs}

We start by introducing the notions of combinatorics of words relevant to our study of MZVs and their generalisations.
\begin{defn}
 \begin{itemize}
  \item For a set $\Omega$, we write $\calw_\Omega$ the linear span (over $\R$) of {\bf words} written in the alphabet $\Omega$, that is to say the algebra over $\R$ of non-commutative polynomials with variables in $\Omega$. We also write $\emptyset$ for the empty word.
  \item The {\bf concatenation product} $\sqcup:\calw_\Omega\times\calw_\Omega\mapsto\calw_\Omega$ is defined by
  \begin{align*}
   \emptyset\sqcup w = w\sqcup\emptyset & = w, \\
   (\omega_1\cdots\omega_k)\sqcup(\omega_1'\cdots\omega_n') &  = (\omega_1\cdots\omega_k\omega_1'\cdots\omega_n')
  \end{align*}
  for any word $w$ in $\calw_\Omega$ and letters $\omega_1,\cdots,\omega_k,\omega_1',\cdots,\omega_n'$ in $\Omega$.
  \item Let $\Omega$ be a set. The {\bf shuffle product} is recursively defined by
  \begin{align*}
   \emptyset\shuffle w = w\shuffle\emptyset & = w, \\
   \left((\omega)\sqcup w\right) \shuffle \left((\omega')\sqcup w'\right) & = (\omega)\sqcup\left[w \shuffle \left((\omega')\sqcup w'\right) \right] + (\omega')\sqcup\left[\left((\omega)\sqcup w\right)\shuffle w'\right].
  \end{align*}
  extended by bilinearity to a product $\shuffle:\calw_\Omega\times\calw_\Omega\longrightarrow\calw_\Omega$.
  \item Let $(\Omega,+)$ be a commutative semigroup (i.e. $+$ is associative and commutative). The {\bf stuffle product} is recursively defined by
  \begin{align*}
   \emptyset\stuffle w = w\stuffle\emptyset & = w, \\
   \left((\omega)\sqcup w\right) \stuffle \left((\omega')\sqcup w'\right) & = (\omega)\sqcup\left[w \stuffle \left((\omega')\sqcup w'\right) \right] + (\omega')\sqcup\left[\left((\omega)\sqcup w\right)\stuffle w'\right]+(\omega+\omega')\sqcup[w\stuffle w'].
  \end{align*}
 \end{itemize}
\end{defn}
It is easy to show that the shuffle and the stuffle products are commutative. It is also a standard exercise to check that they are associative.

Another important definition for the theory of MZVs is the one of convergent words.
\begin{defn}
 A word $w\in\calw_{\{x,y\}}$ is called {\bf convergent} if it is either empty or starts with an $x$ and ends with a $y$. We write $\calw_{\{x,y\}}^{\rm conv}$ the set of convergent words in $\calw_{\{x,y\}}$.
 
 A word $w\in\calw_{\N^*}$ is called {\bf convergent} if it is either empty or has not $1$ as its first letter. We write $\calw_{\N^*}^{\rm conv}$ the set of convergent words in $\calw_{\N^*}$.
\end{defn}
A crucial result in the theory of MZVs is then
\begin{thm}
 The set $\calw_{\{x,y\}}^{\rm conv}$ (resp. $\calw_{\N^*}^{\rm conv}$) is a sub-algebra for the shuffle (resp. stuffle) product. The map $\zeta_\shuffle$ (resp. $\zeta_\stuffle$) is an algebra morphism for the shuffle (resp. stuffle) product: for any words $w_1$, $w_2$ in $\calw_{\{x,y\}}^{\rm conv}$ (resp. $\calw_{\N^*}^{\rm conv}$)
 \begin{equation} \label{eq:stuffle_MZVs}
 \zeta_\shuffle(w_1\shuffle w_2)=\zeta_\shuffle(w_1)\zeta_\shuffle( w_2)\qquad\text{(resp. }\zeta_\stuffle(w_1\stuffle w_2)=\zeta_\stuffle(w_1)\zeta_\stuffle( w_2)\quad\text{)}.
\end{equation}
\end{thm}
To rigorously state the last property of MZVs we will generalise to rooted forests, namely Kontsevitch's relation, we need to define one more object.
\begin{defn}
 The {\bf binarisation map} $\s:\calw_{\N^*}\longrightarrow\calw_{\{x,y\}}$ is the linear map defined by $\s(\emptyset)=\emptyset$ and for any non empty word
 \begin{equation*}
  \s(n_1\cdots n_k):=(\underbrace{x\cdots x}_{n_1-1}y\cdots \underbrace{x\cdots x}_{n_k-1}y).
 \end{equation*}
\end{defn}
Let us write some examples of the action of $\s$:
\begin{equation*}
 \s(23)=(xyxxy),\qquad\s(131)=(yxxyy),\qquad\s(4136)=(xxxyyxxyxxxxxy).
\end{equation*}
We then have
\begin{thm}[Kontsevitch's relation]
 $\s$ maps convergent words $\calw_{\N^*}^{\rm conv}$ to convergent words $\calw_{\{x,y\}}^{\rm conv}$ and
 \begin{equation} \label{eq:Kontsevitch}
 \zeta_\stuffle = \zeta_\shuffle\circ\s.
\end{equation}
\end{thm}

\subsection{Arborified zeta values}

We start we recalling some basic definition of graph theory (see for example \cite{Wi96}) which will be useful in the sequel.
\begin{defn}
 \begin{itemize}
  \item A {\bf graph} is a pair of finite sets $G:=(V(G),E(G))$ with $E(G)\subseteq V(G)\times V(G)$. $E(G)$ is the set of edges of the graph and $V(G)$ the set of vertices of the graph. 
  \item A {\bf path} in a graph $G$ is a finite sequence of elements of $V(G)$: $p=(v_1,\cdots,v_n)$ such that for all $i\in[[1,n-1]]$, $(v_i,v_{i+1})$ is an edge of $G$.  
  By convention, there is always a path between a vertex and itself. 
 \end{itemize}
\end{defn}
We have actually defined oriented graphs since we will only work with such graphs. In particular, ``graph'' will always be used for ``oriented graph''.

In this paper, we will be chiefly concerned with rooted trees so let us introduce related vocabulary and notations that will be used in the rest of this document.
\begin{defn}
 \begin{itemize}
  \item For a graph $G=(V(G);E(G))$, let $\preceq$ be the binary relation on $V(G)$ defined by: $v_1\preceq v_2$ if, and only if, it exists a path from $v_1$ to $v_2$. We also denote by $\succeq$ the inverse relation. A {\bf directed acyclic graph} (DAG for short) is a graph such that $(V(G),\preceq)$ is a poset.
  
  \item A {\bf forest} is a DAG such that there is at most one path between two vertices. A {\bf rooted forest} is a forest whose connected components each have a unique minimal element. These elements are called {\bf roots}. A {\bf rooted tree} is a connected rooted forest.
  \item Let $F$ be a rooted forest and $v_1,v_2$ be two vertices of $F$. If $(v_1,v_2)\in E(F)$\footnote{which implies $v_1\preceq v_2$}, then $v_1$ is called the {\bf direct ancestor} of $v_2$ and $v_2$ a {\bf direct descendant} of $v_1$. We write $v_1=a(v_2)$
  \item If a vertex of a forest $F$ has more than one direct descendant it is called a {\bf branching vertex} of $F$. Furthermore, a vertex that is maximal for the partial order $\preceq$ is called a {\bf leaf}.
  \item Let $\Omega$ be a set. A {\bf $\Omega$-decorated rooted forest} is a rooted forest $F$ together with a {\bf decoration map} 
  $d:V(F)\mapsto\Omega$. For a rooted forest $(F,d_F)$ decorated by $\Omega$ and $\omega\in\Omega$, we write $V_\omega(F)\subseteq V(F)$ the set of vertices of $F$ decorated by $\omega$.
  \item Two rooted forests $F$ and $F'$ (resp. decorated rooted forests $(F,d_F)$ and $(F',d_{F'})$) are {\bf isomorphic} if there  a poset isomorphism $f_V:V(F)\longrightarrow V(F')$ (resp. and $d_F = d_{F'}\circ f_V$) exists.
 \end{itemize}
 We write $\calf$ (resp. $\calf_\Omega$) the commutative algebra freely generated by isomorphism classes of rooted forests (resp. by $\Omega$-decorated rooted) with the product given by the concatenation of graphs. We also use $\calt$ and $\calt_\Omega$ for the vector spaces of isomorphism classes of rooted trees and $\Omega$-decorated rooted trees respectively.
\end{defn}
As usual, we always consider isomorphism classes of rooted forests and therefore identify trees and forests with their classes. Furthermore, when there is no need to specify the decoration map we simply write $F$ for a decorated forest $(F,d)$.

Here we will be primarily interested by rooted forests decorated by a set of two elements $\{x,y\}$. As said before, for such a forest $F$ we write $V_x(F)$ and $V_y(F)$ the set of vertices of $F$ decorated by $x$ and $y$ respectively. The next definition characterises the rooted forests to which we will be able to attach a iterated integral. It is taken from \cite{Cl20}.


\begin{defn} \label{def:conv_forest_integral}
 A rooted forest decorated by $\Omega=\{x,y\}$ is {\bf convergent} if all its leaves and branching vertices are decorated by $y$ and all its roots are decorated by $x$. We write $\calf_{\{x,y\}}^{\rm conv}$ the set of convergent rooted forests decorated by $\{x,y\}$.
 
 A rooted forest decorated by $\Omega=\N^*$ is {\bf convergent} if none of its roots are decorated by $1$. We write $\calf_{\N^*}^{\rm conv}$ the set of convergent rooted forests decorated by $\N^*$.
\end{defn}
Notice that this definition implies in particular that the roots of a convergent forest in $\calf_{\{x,y\}}$ cannot be branching vertices. Now we can define the first family of objects that we will study in this paper.

\begin{defn}[\cite{Ma13}]
 Let $(F,d_F)$ be a convergent rooted forest decorated by $\{x,y\}$. The {\bf arborified zeta value} associated to $(F,d_F)$ is defined by
 \begin{equation*}
  \zeta_\shuffle^T(F) := \int_{\Delta_F}\prod_{v\in V(F)} \omega_v(z_v)
 \end{equation*}
 with 
 \begin{align*}
  \omega_v(z_v)=
  \begin{cases}
   & \frac{dz_v}{z_v}\quad\text{if }d_F(v)=x \\
   & \frac{dz_v}{1-z_v}\quad\text{if }d_F(v)= y
  \end{cases}
 \end{align*}
 and $\Delta_F\subset[0,1]^{|V(F)|}$ defined by
 \begin{equation*}
  [0,1]^{|V(F)|}\ni(z_{v_1},\cdots,z_{v_{|V(F)|}})\in\Delta_F:\Longleftrightarrow \left(v_i\preceq v_j \Leftrightarrow z_{v_j}\leq z_{v_i}\right).
 \end{equation*}
\end{defn}
\begin{rk}
 The integral of arborified zeta values can converge even if $F$ is not a convergent forest. In particular, one can relax the  condition that its branching vertices are decorated by $x$ and still have a convergent integral. This is clear from the proof of convergence of arborified zeta values made in \cite[Definition-Proposition 4.8 and Lemma 4.13]{Cl20}.
 
 We focus on these forests for two reasons. First, these convergent rooted forests are the image of the branched binarisation map $\s^T:\calf_{\N^*}\longrightarrow\calf_{\{x,y\}}$ (see \cite[Definition A.1]{Cl20}). Second, for these forests, the associated arborified zeta values can be written as a multiple series over a domain given by the rooted forest they are attached to. This is one of the main results of this paper.
\end{rk}
Notice that $\zeta_\shuffle^T(F)$ should be written $\zeta_\shuffle^T(F,d_F)$ to be completely rigorous. We drop the reference to the decoration map in order to simplify notations. Notice further that we inverse the partial order of vertices of the forest and the partial order on the attached integration variable to follow \cite{Cl20}. This is simply done to simplify notations, and the other choice was made, for example in \cite{Ma13}.

Before we proceed further, let us recall that for any $\omega\in\Omega$ the {\bf grafting operator} $B_+^\omega:\calf_{\Omega}\longrightarrow\calt_\Omega$ which is a linear operator that, to any rooted forest $F=T_1\cdots T_k$, associates the decorated tree obtained from $F$ by adding a root decorated by $\omega$ linked to each root of $T_i$ for $i$ going from $1$ to $k$. We can now introduce the shuffle of trees that have been defined in \cite{Cl20}.
\begin{defn}
 Let $\Omega$ be a set. The {\bf shuffle product on forests} $\shuffle^T$ of two forests $F$ and 
 $F'$ is defined recursively on $|F|+|F'|$. 
 
 If $|F|+|F'|=0$ (and thus $F=F'=\emptyset$), we set $\emptyset\shuffle^T \emptyset = \emptyset$.
 
 For $N\in\N$, assume the shuffle products of forests has been defined on every forests $f,f'$ such that $|f|+|f'|\leq N$. Then for any two forests $F,F'$ such that $|F|+|F'|=N+1$;
 \begin{itemize}
  \item If $F'=\emptyset$, set $\emptyset\shuffle^T F=F\shuffle^T \emptyset=F$;
  \item If $F$ or $F'$ is not a tree, then we can write $F$ and $f$ uniquely as a concatenation of trees: $F=T_1\cdots T_k$ and $F'=t_1\cdots t_n$ with the $T_i$s and $t_j$s 
  nonempty, $k+n\geq3$ and set 
  \begin{equation*}
   F\shuffle^T F' = \frac{1}{kn}\sum_{i=1}^k\sum_{j=1}^n\left((T_i\shuffle^T t_j)T_1\cdots\widehat{T_i}\cdots T_n t_1\cdots\widehat{t_j}\cdots t_k\right)
  \end{equation*}
  where $T_1\cdots\widehat{T_i}\cdots T_n$ stands for the concatenation of the trees $T_1,\cdots,T_n$ without the tree $T_i$.
  \item If $F=T = B_+^a(f)$ and $F'=T'=B_+^{a'}(f')$ are two nonempty trees, we set 
  \begin{equation*}
   T\shuffle^T T' = B_+^a(f\shuffle^T T') + B_+^{a'}(T\shuffle^T f')
  \end{equation*}
 \end{itemize}
\end{defn}
It is easy to show that $\shuffle^T$ is commutative, but it is not associative (see \cite[Counterexample 5.7]{Cl20}). However an important result of \cite{Cl20} is
\begin{thm} (\cite[Theorem 5.11]{Cl20})
 $\calf_{\{x,y\}}^{\rm conv}$ is a subalgebra for the shuffle product of forests $\shuffle^T$. Furthermore, for any convergent forests $F_1$ and $F_2$ in $\calf_{\{x,y\}}^{\rm conv}$ we have
 \begin{equation*}
  \zeta_\shuffle^T(F_1\shuffle^T F_2)=\zeta_\shuffle^T(F_1)\zeta_\shuffle^T(F_2).
 \end{equation*}
\end{thm}
In other words, the linear map $\zeta_\shuffle^T:\calf_{\{x,y\}}^{\rm conv}\longrightarrow\R$ which sends a convergent forest $F\in\calf_{\{x,y\}}^{\rm conv}$ to $\zeta_\shuffle^T(F)$ is an algebra morphism for the shuffle product of forests.

We need to introduce one more notion in this section.


\begin{defn} \label{defn:flattening}
 Let $\Omega$ be a set, the {\bf flattening map} $fl_0:\mathcal{F}_\Omega\longrightarrow\calw_\Omega$  from the algebra of rooted forests decorated by $\Omega$ and 
 the algebra of words $\mathcal{W}_\Omega$  written in the alphabet $\Omega$ is recursively defined by
 \begin{equation*}
  fl_0(\emptyset)=\emptyset,\quad fl_0(F_1F_2) = fl_0(F_1)\shuffle fl_0(F_2),\quad fl_0(B_+^\omega(F))=(\omega)\sqcup fl_0(F)
 \end{equation*}
 extended by linearity to a map on $\calf_\Omega$.
\end{defn}
\begin{rk}
 In \cite{Ma13}, the flattening map is called the (simple) arborification, following Ecalle in \cite{Ecalle}. It also appears in \cite{Ya20}. We chose instead to follow \cite{CGPZ18} and \cite{Cl20}. Furthermore, together with its weighted versions, it can be defined from a universal property of the algebra of rooted forests (see \cite[Definition 2.14]{Cl20}). Since this formulation requires to define more structures, we have opted here for this more pedestrian approach in order to keep the reminders of this section within reasonable length.
\end{rk}
The following result was shown in \cite{Cl20} (Lemma 4.9 and Theorem 4.15) although a version for divergent AZVs was already presented in the earlier work \cite{CGPZ18}. It refines \cite[Corollary 2.4]{Ya14} and was also stated in \cite{Ya20}. 
\begin{thm} \label{thm:flattening}
 The flattening map $fl_0$ maps convergent forests to convergent words and for any convergent rooted forest $(F,d_F)\in\calf^{\rm conv}_{\{x,y\}}$, the arborified zeta value $\zeta_\shuffle^T(F)$ is a finite linear combination of MZVs with rational coefficients given by
 \begin{equation*}
  \zeta_\shuffle^T(F) = \zeta_\shuffle(fl_0(F))
 \end{equation*}
 with $\zeta_\shuffle$ the MZV map defined in Equation \eqref{eq:integral_shuffle_MZV}.
\end{thm}
\begin{rk}
 In \cite{Cl20}, an equivalent of the stuffle MZVs (Equation \ref{eq:sum_stuffle_MZV}) for rooted forests decorated by $\N^*$ was also build and studied. It was shown this object has properties similar to $\zeta_\shuffle^T$. Namely, it is an algebra morphism for a stuffle product of rooted forest and can be related to stuffle MZVs via a contracting flattening map. However, it was also shown that these two AZVs are \emph{not} related by the most natural generalisation of Kontsevitch's relation \eqref{eq:Kontsevitch}. This is why we do not introduce these objects here and will instead replace them by a new generalisation of MVZs for which Kontsevitch's relation holds.
\end{rk}

\subsection{Series representation of AZVs}

In order to prove the existence of a series representation for AZVs, we need to introduce a few more notions.
\begin{defn}
 A {\bf segment} of a rooted forest $(F,d_F)$ decorated by $\{x,y\}$ is a non-empty path $s_v=(v_1,\cdots,v_n=v)$ such that $d_F(v_n)=y$, $d_F(v_i)=x$ for any $i$ in $\{1,\cdots,n-1\}$ and $d(a(v_1))=y$, with $a(v_1)$ the direct ancestor of $v_1$. We call the number $n$ the {\bf length} of the segment $s_v$ and write it $|s_v|$.
 
 We write $S(F):=\{s_v|v\in V_y(F)\}$ the set of segments of a rooted forest $F$.
\end{defn}
In words: a segment $s_v$ of a rooted forest is a path in this rooted forest, which ends at the vertex $v$ decorated by $y$ and starts just above the first ancestor of $v$ being also decorated by $y$.

The set $S(F)$ inherits a poset structure from the poset structure of $V(F)$. We also denote this partial order relation by $\preceq$: $s_v\preceq s_{v'}:\Longleftrightarrow v\preceq v'$. This allows us to define the depth of a segment.
\begin{defn}
 The {\bf depths} of the segments of a decorated rooted forest $F$ decorated by $\{x,y\}$ are recursively defined by:
 \begin{itemize}
  \item depth$(s_v)=0$ if $v$ is a leaf of $F$,
  \item depth$(s_v)=\max\{{\rm depth}(s_{v'})|v'\neq v\wedge s_v\preceq s_{v'}\}+1$.
 \end{itemize}
 We also set $N_F:=\max\{{\rm depth}(s_v)|s_v\in S(F)\}$ the maximal depth of a segment of $F$. For any $n\in\{0,\cdots,N_F\}$ we set 
 \begin{equation*}
  S_n(F) := \{s_v\in S(F)|{\rm depth}(s_v)=n\},
 \end{equation*}
  we further set:
 \begin{equation*}
  S^n(F):=\bigcup_{i=0}^n S_i(F)
 \end{equation*}
 for any $n$ in $\{0,\cdots,N_F\}$. For any such $n$ we also write
 \begin{equation*}
  ||S_n(F)||:=\sum_{s_v\in S_n(F)} |s_v| \quad\text{and}\quad ||S^n(F)||:=\sum_{s_v\in S^n(F)} |s_v|.
 \end{equation*}

\end{defn}
Notice that depth of segments of a rooted forest are well-defined because we have taken our forests to be finite.

We are now able to prove the following
\begin{thm} \label{thm:integral_sum}
 For any convergent forest $F$ the corresponding arborified zeta values admits the following series representation:
 \begin{equation*}
  \zeta_\shuffle^T(F) = \sum_{\substack{ n_v\geq1\\v\in V_y(F)}}\prod_{v\in V_y(F)}\left(\sum_{\substack{v'\in V_y(F)\\ v'\succeq v}}n_{v'}\right)^{-|s_v|}
 \end{equation*}
 where for $s_v=(v_1,\cdots,v_n)$ we set $|s_v|=n$ and $v\in V_y(F)$ in the first sum means that this series has a summation variable for each $v\in V_y(F)$.
\end{thm}
\begin{proof}
 Let $F$ be a convergent forest. Since the map $F\mapsto\zeta_\shuffle^T(F)$ is an algebra morphism for the concatenation product of trees, it is enough to show that the theorem holds for $F$ a rooted tree. Thus we can assume without loss of generality that $F$ is a rooted tree.
 
 Let $N_F$ be the maximal depth of the segments of this tree. If $N_F=0$, the theorem reduces to the usual series representation of a Riemann zeta. 
 
 If $N_F\geq1$, since we are working with convergent integrals we can use Fubini's theorem to regroup integrations of the segment. This mean that we can write the arborified zeta associated to $F$ as
 \begin{equation} \label{eq:step_zero}
  \zeta_\shuffle^T(F) = \int_{\Delta_F}\prod_{n=0}^{N_F}\prod_{s_v\in S_n(F)} d\omega_{s_v}
 \end{equation}
 where, for $s_v=(w_1,\cdots,w_{p})\in S_n(F)$ we have set 
 \begin{equation*}
  d\omega_{s_v}:=\frac{dz_{w_1}}{z_{w_1}}\cdots\frac{dz_{w_{p-1}}}{z_{w_{p-1}}}\frac{dz_{w_p}}{1-z_{w_p}}
 \end{equation*}
 (with an obvious abuse of notation if $p=1$). Now recall that for any $A\in\R_+$, $p\geq1$ and any $Z\in[0,1]$ ($Z\neq 1$ if $p=1$) we have
 \begin{equation} \label{eq:lemma_trad}
  \int_{0\leq z_p\leq\cdots\leq z_1\leq Z}\frac{dz_1}{z_1}\cdots\frac{dz_{p-1}}{z_{p-1}}\frac{dz_p}{1-z_p} (z_p)^A = \sum_{n=1}^{+\infty} \frac{Z^{n+A}}{(n+A)^p}
 \end{equation}
 (again with obvious abuses of notations when $p=1$). This classical result follows from the theorem of dominated convergence and the Taylor expansion of the function $x\mapsto (1-x)^{-1}$.
 
 We can now use \eqref{eq:lemma_trad} with $A=0$ in \eqref{eq:step_zero} to integrate all the variable attached to vertices belonging to a segment of depth zero. We obtain 
 \begin{equation} \label{eq:step:one}
  \zeta_\shuffle^T(F) = \sum_{\substack{n_v=1, \\ v\in V_y(F)| s_v\in S_0(F)}}^{+\infty} \prod_{v\in V_y(F)|s_v\in S_0(F)}(n_v)^{-|s_v|}\int_{\Delta_F\setminus S_0(F)}\prod_{n=1}^{N_F}\prod_{s_v\in S_n(F)} d\omega_{s_v}\prod_{v\in V_y(F)|s_v\in S_1(F)}(z_v)^{\sum_{\substack{v'\in V_y(F)\\v'\succ v}}n_{v'}}
 \end{equation}
 where $v'\succ v$ means that $v'$ is a descendant of $v$ that is distinct from $v$ 
 and where we have set 
 \begin{align*}
  & [0,1]^{|V(F)|-||S_0(F)||}\ni(z_{v_1},\cdots,z_{v_p})\in\Delta_F\setminus S_0(F) \\
  :\Longleftrightarrow & \left(\{v_1,\cdots,v_p\}=V(F)\setminus\{v'\in s_v|s_v \in S_0(F)\}~\wedge~(v_i\preceq v_j \Leftrightarrow z_{v_j}\leq z_{v_i})\right).
 \end{align*}
 Now for any $k\in\{0,\cdots,N_F-1\}$ we set 
 \begin{align*}
  & [0,1]^{|V(F)|-||S^k(F)||}\ni(z_{v_1},\cdots,z_{v_p})\in\Delta_F\setminus S^k(F) \\
  :\Longleftrightarrow & \left(\{v_1,\cdots,v_p\}=V(F)\setminus \{v'\in s_v|s_v \in S^k(F)\}~\wedge~(v_i\preceq v_j \Leftrightarrow z_{v_j}\leq z_{v_i})\right)
 \end{align*}
 (notice that we replaced $S_0(F)$ by $S^k(F)$). Let us prove by induction over $k$ that for any $k\in \{0,\cdots,N_F-1\}$ we have
 \begin{align} \label{eq:step_rec}
  \zeta_\shuffle^T(F) & = \sum_{\substack{n_v=1, \\ v\in V_y(F)| s_v\in S^k(F)}}^{+\infty} \prod_{v\in V_y(F)|s_v\in S^k(F)}\left(\sum_{\substack{v'\in S^k(F) \\ v'\succeq v}}n_{v'}\right)^{-|s_v|} \\ 
   & \times \underbrace{\int_{\Delta_F\setminus S^k(F)}\prod_{n=k+1}^{N_F}\prod_{s_v\in S_n(F)} d\omega_{s_v}\prod_{v\in V_y(F)|s_v\in S_{k+1}(F)}(z_v)^{\sum_{\substack{v'\in V_y(F) \\v'\succ v}}n_{v'}}}_{=:I_{F,k}}.  \nonumber
 \end{align}
 First, since $S_0(F)=S^0(F)$ and $\Delta_F\setminus S_0(F)=\Delta_F\setminus S^0(F)$, Equation \eqref{eq:step:one} is exactly Equation \eqref{eq:step_rec} for $k=0$. Then if $N_F=1$, we have proven Equation \eqref{eq:step_rec} in every cases of interests. If $N_F\geq2$ then let us assume that Equation \eqref{eq:step_rec} holds for $k\in\{0,\cdots,N_F-2\}$. We then have, once again from the Taylor expansion of the function $x\to(1-x)^{-1}$ and the dominated convergence theorem
 \begin{align*}
  I_{F,k} & = \int_{\Delta_F\setminus S^k(F)}\left(\prod_{n=k+2}^{N_F}\prod_{s_v\in S_n(F)} d\omega_{s_v}\right)\left(\prod_{s_v\in S_{k+1}(F)} d\omega_{s_v}\right)\prod_{v\in V_y(F)|s_v\in S_{k+1}(F)}(z_v)^{\sum_{\substack{v'\in V_y(F) \\v'\succ v}}n_{v'}} \\
  & = \sum_{\substack{n_v=0, \\ v\in V_y(F)| s_v\in S_{k+1}(F)}}^{+\infty} \int_{\Delta_F\setminus S^k(F)}\left(\prod_{n=k+2}^{N_F}\prod_{s_v\in S_n(F)} d\omega_{s_v}\right)  \\
  \times & \prod_{\substack{v\in V_y(F)| \\ s_v=(\vec{\alpha},v)\in S_{k+1}(F)}} \frac{d\vec{z}_{\vec{\alpha}}}{\vec{z}_{\vec\alpha}}dz_v(z_v)^{n_v} \prod_{v\in V_y(F)|s_v\in S_{k+1}(F)}(z_v)^{\sum_{\substack{v'\in V_y(F) \\v'\succ v}}n_{v'}}
 \end{align*}
 with the obvious notation that $\frac{d\vec{z}_{\vec{\alpha}}}{\vec{z}_{\vec\alpha}}$ is a product of $dz/z$. We can now merge the two last products to obtain
 \begin{equation*}
  I_{F,k} = \sum_{\substack{n_v=0, \\ v\in V_y(F)| s_v\in S_{k+1}(F)}}^{+\infty} \int_{\Delta_F\setminus S^k(F)}\left(\prod_{n=k+2}^{N_F}\prod_{s_v\in S_n(F)} d\omega_{s_v}\right) 
  \prod_{\substack{v\in V_y(F)| \\ s_v=(\vec{\alpha},v)\in S_{k+1}(F)}} \frac{d\vec{z}_{\vec{\alpha}}}{\vec{z}_{\vec\alpha}}dz_v(z_v)^{\sum_{\substack{v'\in V_y(F) \\v'\succeq v}}n_{v'}}.
 \end{equation*}
 Finally integrating the variable attached to vertices belonging to segments of depth $k$ and switching each of the summation variable by one we obtain
 \begin{equation*}
  I_{F,k} = \sum_{\substack{n_v=1, \\ v\in V_y(F)| s_v\in S_{k+1}(F)}}^{+\infty} \prod_{v\in V_y(F)|s_v\in S_{k+1}(F)}\left(\sum_{\substack{v'\in V_y(F) \\ v'\succeq v}}n_{v'}\right)^{-|s_v|}
  \underbrace{\int_{\Delta_F\setminus S^{k+1}(F)}\left(\prod_{n=k+2}^{N_F}\prod_{s_v\in S_n(F)} d\omega_{s_v}\right)}_{I_{F,k+1}}.
 \end{equation*}
 Notice that this whole computation consisted essentially into using Formula \eqref{eq:lemma_trad} for each of the segments of depth exactly $k$. In any case, plugging this expression for $I_{F,k}$ back into \eqref{eq:step_rec} we obtain exactly the same equation with $k$ replaced by $k+1$. So, by a finite induction we have proven Equation \eqref{eq:step_rec} for any $k\in\{0,\cdots,N_F\}$ for any value of $N_F\geq1$.
 
 Since we have assumed $F$ to be a connected rooted forest (i.e. a rooted tree), $F$ has exactly one segment $s_v$ of maximal depth $N_F$. Furthermore, since $F$ is convergent, we have $l:=|s_v|\geq2$. Thus, after a relabelling of $s_v$:
 \begin{equation*}
  s_v=(1,\cdots,l=v),
 \end{equation*}
 Equation \eqref{eq:step_rec} with $k=N_F-1$ reads
 \begin{align*}
  \zeta_\shuffle^T(F) & = \sum_{\substack{n_v=1, \\ v\in V_y(F)| s_v\in S^{N_F-1}(F)}}^{+\infty} \prod_{v\in s_v|s_v\in S^{N_F-1}(F)}\left(\sum_{\substack{v'\in S^{N_F-1}(F) \\ v'\succeq v}}n_{v'}\right)^{-|s_v|} \\ 
   & \times \int_{0\leq z_1<\cdots<z_l\leq1} \frac{dz_1}{z_1}\cdots\frac{dz_{l-1}}{z_{l-1}}\frac{dz_l}{1-z_l}\left(z_l\right)^{\sum_{\substack{v'\in V_y(F) \\ v'\neq v}}n_{v'}}
 \end{align*}
 Using once again Formula \eqref{eq:lemma_trad} with $Z=1$, $p=l\geq2$ and $A=\sum_{\substack{v'\in V_y(F) \\ v'\neq v}}n_{v'}$ we obtain the statement of the theorem once we write all the sums together.

\end{proof}

\begin{rk} 
 This series representation of AZVs is \emph{not} the stuffle AZVs of \cite{Cl20} which were defined as iterated series. Theorem A.4 of \cite{Cl20} implies that these two series applied to the same forest give in general different values. Instead, this series representation of AZVs defined by iterated integrals should be seen as a new generalisation of MZVs defined as iterated series. It is not the purpose of this paper to explore their algebraic structures and we left that for further research. Instead, we will focus on relating this new generalisation to other generalised MZVs.
\end{rk}

\section{Tree zeta values} \label{section:two_bis}

\subsection{Definition and first properties}

Theorem \ref{thm:integral_sum} motivates a new generalisation of MZVs to trees. Let us first state the definition of these new iterated sums without taking care of their convergence.
\begin{defn} \label{defn:tree_zeta_values}
 For a, $\N^*$-decorated rooted forest $F$, whenever it exists, let 
 \begin{equation} \label{eq:TZV}
  \zeta^t(F) := \sum_{\substack{ n_v\geq1\\v\in V(F)}}\prod_{v\in V(F)}\left(\sum_{\substack{v'\in V(F)\\ v'\succeq v}}n_{v'}\right)^{-\alpha_v}
 \end{equation}
 (with $\alpha_v=d_F(v)\in\N^*$ the decoration of the vertex $v$ and, as before, $v\in V(F)$ in the first sum means that this series has a summation variable for each $v\in V(F)$) be the {\bf tree zeta value} (TZV) associated to $F$. $\zeta^t$ is then extended by linearity to a map defined on a subset of $\calf_{\N^*}$.
\end{defn}
\begin{rk} 
In a private communication, F. Zerbini  suggested that tree zeta values could be of interest, and in particular that they likely were MZVs. I am thankful for this input.
\end{rk}
Here are some examples of TZVs:
\begin{equation*}
 \zeta^t(\tdun{n})=\zeta_\stuffle(n),\qquad \zeta^t(\tdtroisun{$2$}{$1$}{$2$})=\sum_{n_1,n_2,n_3\geq1}\frac{1}{(n_1+n_2+n_3)^2n_2(n_3)^2}.
\end{equation*}
\begin{rk}
 A simple change of summing variables in the interated series of Equation \eqref{eq:sum_stuffle_MZV} allows to rewrite stuffle MZVs as
 \begin{equation*}
  \zeta_\stuffle(n_1\cdots n_k)=\sum_{m_1,\cdots,m_k\geq1}\frac{1}{(m_1+\cdots+m_k)^{n_1}(m_2+\cdots+m_k)^{n_2}\cdots(m_k)^{n_k}}.
 \end{equation*}
 This is the expression of a tree zeta value associated to a ladder tree decorated by $n_1,\cdots,n_k$. Therefore, TZVs are indeed a new generalisation of MZVs to rooted forests.
\end{rk}

To study the convergences of tree zeta values, let us recall some definitions of \cite{Cl20}
\begin{defn} \cite[Definition A.1]{Cl20}
 The {\bf branched binarisation map} is the linear map
 $\fraks^T :\calf_{\N^*}\longrightarrow\calf_{\{x,y\}}$ recursively defined by
 \begin{align*}
  \begin{cases}
   & \fraks^T(\emptyset) = \emptyset \\
   & \fraks^T(T_1\cdots T_n) = \fraks^T(T_1)\cdots\fraks^T(T_n) \\
   & \fraks^T(B_+^n(F)) = (B_+^x)^{\circ(n-1)}\circ B_+^y(\fraks^T(F))
  \end{cases}
 \end{align*}
 and extended by linearity to a map of $\calf_{\N^*}$.
\end{defn}
In \cite{Cl20}, a more algebraic definition of $\fraks^T$ was given using a universal property of the algebra $\calf_{\N^*}$. The above definition is more pedestrian but requires less background. In order to clarify the action of $\fraks^T$, let us write down some examples of the action of $\fraks^T$ on simple trees.
 \begin{equation*}
  \fraks^T(\tdun{1}) = \tdun{y} \qquad  \fraks^T(\tdun{2}) = \tddeux{$x$}{$y$} \qquad \fraks^T\left(\tdtroisun{1}{1}{2}\right) = \tdquatredeux{$y$}{$y$}{$x$}{$y$} \qquad
  \fraks^T\left(\tdtroisun{$2$}{$1$}{$2$}\right) = \tcinqonze{$x$}{$y$}{$x$}{$y$}{$y$}.
 \end{equation*}
 One further definition of importance is the following
 \begin{defn} \label{def:conv_forest_series} \cite[Definition 3.14]{Cl20}
  A $\N^*$-decorated tree is called {\bf convergent} if and only if it is either empty or has  its root decorated by $n\geq2$. A $\N^*$-decorated forest is called {\bf convergent} if and only if it is a disjoint union of convergent trees. We write $\calf_{\N^*}^{\rm conv}$ the set of convergent rooted forests decorated by $\N^*$.
 \end{defn}
 The following trivial characteristics of the map $\fraks^T$ were shown in \cite{Cl20}:
 \begin{lem} \label{lem:bin_map_properties} \cite[Lemma A.3]{Cl20}
  The branched binarisation map $\fraks^T$, when restricted to convergent forests, is a bijection and maps convergent forests to convergent forests:
  \begin{equation*}
   \fraks^T\left(\calf_{\N^*}^{\rm conv}\right) = \calf_{\{x,y\}}^{\rm conv}.
  \end{equation*}
 \end{lem}
 The following simple yet important result justifies the definition of tree zeta values.
 \begin{prop} \label{prop:crucial}
  For any convergent $\N^*$-decorated rooted forest $F$, the tree zeta value $\zeta^t(F)$ associated to $F$ is convergent and is equal to the arborified zeta values associated to the convergent $\{x,y\}$-decorated rooted forest $\fraks^T(F)$:
  \begin{equation*}
   \zeta^t(F) = \zeta^T_\shuffle(\fraks^T(F)).
  \end{equation*}
 \end{prop}
 \begin{proof}
  The result follows from the simple observation that, for any convergent $\{x,y\}$-decorated rooted forest $f$ the series representation of $\zeta_\shuffle^T(f)$ given by Theorem \ref{thm:integral_sum} is precisely the tree zeta value $\zeta^t((\fraks^T)^{-1}(f))$. The convergence of $\zeta^t(F)$ for any convergent $\N^*$-decorated rooted forest $F$ then follows from the facts that $\fraks^T$ is a one to one map between the two sets of convergent rooted forests and that the arborified zeta value $\zeta_\shuffle^T(f)$ converges for any convergent $\{x,y\}$-decorated rooted forest $f$.
 \end{proof}
 In particular, when $F$ is a ladder tree, Kontsevitch's relation (Equation \eqref{eq:Kontsevitch}) gives us
 \begin{equation} \label{eq:TZVs_stuffle_words}
  \zeta^t(\iota(w))=\zeta_\stuffle(w)
 \end{equation}
 for any convergent word $w$; where $\iota$ is the canonical map sending words to ladder trees.
 
 From this Proposition, one easily derives important properties of tree zeta values from the properties of branched zeta values.
 \begin{thm} \label{thm:tree_zeta_MZVs}
  \begin{itemize}
   \item The map $\zeta^t:F\mapsto\zeta^t(F)$ is an algebra morphism from convergent $\N^*$-decorated rooted forests to $\R$ for the concatenation product of forests.
   \item For any convergent $\N^*$-decorated rooted forests $F$, the tree zeta value $\zeta^t(F)$ is a $\Q$-linear combination of MZVs, given by
   \begin{equation*}
    \zeta^t(F) = (\zeta_\shuffle\circ fl_0\circ\fraks^T)(F).
   \end{equation*}
  \end{itemize}
 \end{thm}
 \begin{proof}
  Both points follow directly from Proposition \ref{prop:crucial}
  \begin{itemize}
   \item The first point follows from the fact that $\zeta^t=\zeta_\shuffle^T\circ\fraks^T$ together with the fact that both $\zeta_\shuffle^T$ and $\fraks^T$ are algebra morphisms for the concatenation product of rooted forests.
   \item The second point follows from the same relation $\zeta^t=\zeta_\shuffle^T\circ\fraks^T$ together with Theorem \ref{thm:flattening} which states that $\zeta_\shuffle^T=\zeta_\shuffle\circ fl$.
  \end{itemize}
 \end{proof}
 \begin{rk}
  Finite sums similar to TZVs were studied in \cite{On16} in the context of finite MZVs. In particular, \cite[Theorem 1.4]{On16} is the equivalent to Theorem \ref{thm:tree_zeta_MZVs} in the context of finite sums.
 \end{rk}
  Notice that $\zeta^t$ is \emph{not} an algebra morphism for the shuffle product of trees, since $\fraks^T$ is not. It can also be checked on computations that it neither is an algebra morphism for any of the stuffle products of trees (see \cite[Definition 5.1]{Cl20}).  The rest of this Section is dedicated to the study of the algebraic properties of TZVs.

 \subsection{The yew product}
 
 The $\yew$ product is an alternative product which will have interesting applications to TZVs as well as their applications to CZVs.
 \begin{defn} \label{defn:yew}
  The $\yew$ product (read ``yew'') is a product $\yew:\calf_{\N^*}\otimes\calf_{N^*}\longrightarrow\calf_{\N^*}$ defined by
  \begin{equation*}
   \yew=(\s^T)^{-1}\circ\shuffle^T\circ(\s^T\otimes\s^T).
  \end{equation*}
 \end{defn}
 \begin{rk} \label{rk:yew_words}
  In some applications, we will use the $\yew$ product on words: for two words $w$ and $w'$ we will write $w\yew w'$ instead of $\iota^{-1}(\iota(w)\yew\iota(w'))$ (with $\iota$ the canonical injection of words into rooted forests). We allow ourselves to make this small abuse of notations as it should not give rise to any confusion and will allow to greatly lighten the notations.
 \end{rk}

 We start by a simple but important property of the $\yew$ product.
 \begin{prop}
  The $\yew$ product is commutative but not associative.
 \end{prop}
 \begin{proof}
  The commutativity directly follows from the commutativity of $\shuffle^T$. The non associativity as well and can be check on examples: take $F_1 = \tdun{2} \tdun{2}$, $F_2 = \tdun{2}$ and $F_3 = \tdun{2}$. After some computations we obtain:
	\begin{align*}
	  (F_1 \yew F_2) \yew F_3 &- F_1 \yew (F_2 \yew F_3) = 
	 2 \times \tddeux{2}{2} \tddeux{2}{2} + 8 \times \tddeux{2}{2} \tddeux{3}{1} + 8 \times \tddeux{3}{1} \tddeux{3}{1} \\
	  &- \left[ 3 \times \tdun{2} \tdtrois{2}{2}{2} + 6 \times \tdun{2} \tdtrois{3}{1}{2} + 6 \times \tdun{2} \tdtrois{2}{3}{1}
+ 12 \times \tdun{2} \tdtrois{3}{2}{1} + 18 \times \tdun{2} \tdtrois{4}{1}{1} \right].
	\end{align*}
 \end{proof}
 Let us know give an explicit formula for the $\yew$ product. We start with two technical Lemmas about the shuffle product of rooted forests decorated by $\{x,y\}$.
 \begin{lem} \label{lem:lem de lem pour yew}
  For any rooted forests $(f,g)\in\calf_{\{x,y\}}^2$ and any $n\in\N$ we have:
  \[ \big( (B_+^x)^{\circ n}\circ  B_+^y(f) \big) \shuffle^T B_+^y(g) 
	=  (B_+^x)^{\circ n}\circ  B_+^y \big[ f \shuffle^T B_+^y(g) \big] + \sum_{j=0}^n (B_+^x)^{\circ j}\circ  B_+^y \Big[ \big( (B_+^x)^{\circ(n-j)}  B_+^y(f) \big) \shuffle^T g \Big].  \]
 \end{lem}
\begin{proof}
 We prove this result by induction on $n$. For $n=0$, the lemma holds by definition of the shuffle product $\shuffle^T$. Let us assume it holds for some $n\in\N$. We then have
 \begin{align*}
	& \big( (B_+^x)^{\circ(n+1)}\circ B_+^xy(f) \big) \shuffle^T B_+^y(g) = B_+^x\circ \Big[ \big( (B_+^x)^{\circ n}\circ B_+^y (f) \big) \shuffle^T B_+^y(g) \Big] + B_+^y \Big[ \big( (B_+^x)^{\circ(n+1)}\circ  B_+^y (f) \big) \shuffle^T g \Big]  \\
		= & B_+^x \bigg( (B_+^x)^{\circ n}\circ B_+^y \big[ f \shuffle^T B_+^y(g) \big] + \sum_{j=0}^a (B_+^x)^{\circ j}\circ B_+^y \Big[ \big( (B_+^x)^{\circ(n-j)}\circ B_+^y(f) \big) \shuffle^T g \Big] \bigg)
        +  B_+^y \Big[ \big( (B_+^x)^{\circ(n+1)}\circ B_+^y (f) \big) \shuffle^T g \Big] \\
        & \hspace{12cm}\qquad \llcorner \text{ by the induction hypothesis} \\
		&= (B_+^x)^{\circ(n+1)}\circ B_+^y \big[ f \shuffle^T B_+^y(g) \big] + \sum_{j=0}^{n+1} (B_+^x)^{\circ j}\circ B_+^y \Big[ \big( (B_+^x)^{\circ(n+1-j)} \circ B_+^y(f) \big) \shuffle^T g \Big]
	\end{align*}
	which conclude the induction and the proof.
\end{proof}
We generalise this result:
\begin{lem} \label{lem:yew_inter}
 For any rooted forests $(f,g)\in\calf_{\{x,y\}}^2$ and any $(n,m)\in\N^2$ we have:
	\begin{align*}
	\big( (B_+^x)^{\circ n}\circ  B_+^y(f) \big) \shuffle^T \big( (B_+^x)^{\circ m}\circ  B_+^y(g) \big) &
	=  \sum_{i=0}^{m} \binom{n+i}{i} (B_+^x)^{\circ(n+i)}\circ B_+^y \left[ f \shuffle^T \big( (B_+^x)^{\circ(m-i)}\circ B_+^y(g) \big) \right] \\
	 +& \sum_{j=0}^{n} \binom{m+j}{j} (B_+^x)^{m+j}\circ B_+^y \left[ \big( (B_+^x)^{\circ(n-j)}\circ B_+^y(f) \big) \shuffle^T g \right].
	\end{align*}
\end{lem}
\begin{proof}
 We prove this result by induction on $n+m$. If $n+m=0$, then it is holds by definition of the shuffle product of rooted trees. Assume that the result hold for some $N=n+m$. Then for any $(n,m)\in\N^2$ such that $n+m=N+1$, if $n=0$ or $m=0$ the result reduces to the previous Lemma \ref{lem:lem de lem pour yew} (eventually using the commutativity of $\shuffle^T$). For $n\neq 0$ and $m\neq0$ we have
 \begin{align*}
  \big( (B_+^x)^{\circ n}\circ B_+^y(f) \big) \shuffle^T \big( (B_+^x)^{\circ m}\circ B_+^y(g) \big) & = B_+^x\circ \Big( \big( (B_+^x)^{\circ(n-1)}\circ B_+^y(f) \big) \shuffle^T \big( (B_+^x)^{\circ m}\circ B_+^y(g) \big) \Big) \\ 
	 + & B_+^x\circ \Big( \big( (B_+^x)^{\circ n}\circ B_+^y(f) \big) \shuffle^T \big( (B_+^x)^{\circ(m-1)}\circ B_+^y (g) \big) \Big).
 \end{align*}
 Using the induction hypothesis we obtain:
 \begin{align*}
  & B_+^x\circ \Big( \big( (B_+^x)^{\circ(n-1)}\circ B_+^y(f) \big) \shuffle^T \big( (B_+^x)^{\circ m}\circ B_+^y(g) \big) \Big) \\
  = & B_+^x \Bigg( 
  \sum_{i=0}^{m} \binom{n-1+i}{i} (B_+^x)^{\circ(n-1+i)}\circ B_+^y \left[ f \shuffle^T \big( (B_+^x)^{\circ(m-i)}\circ B_+^y(g) \big) \right] \\
  & \hspace{5cm}+ \sum_{j=0}^{n-1} \binom{m+j}{j} (B_+^x)^{m+i}\circ B_+^y \left[ \big( (B_+^x)^{\circ(n-1-j)}\circ B_+^y(f) \big) \shuffle^T g \right] \Bigg) \\
  = & \sum_{i=0}^{m} \binom{n-1+i}{i} (B_+^x)^{\circ(n+i)}\circ B_+^y \left[ f \shuffle^T \big( (B_+^x)^{\circ(m-i)}\circ B_+^y(g) \big) \right] \\
	&\hspace{5cm}
	+ \sum_{j=1}^{n} \binom{m-1+j}{j-1} (B_+^x)^{\circ(m+j)}\circ B_+^y \left[ \big( (B_+^x)^{n-j}\circ B_+^y(f) \big) \shuffle^T g \right].
 \end{align*}
 Similarly we obtain:
 \begin{align*}
  & B_+^x\circ \Big( \big( (B_+^x)^{\circ n}\circ B_+^y(f) \big) \shuffle^T \big( (B_+^x)^{\circ(m-1)}\circ B_+^y (g) \big) \\
  = & \sum_{i=1}^{m} \binom{n-1+i}{i-1} (B_+^x)^{\circ(n+i)}\circ B_+^y \left[ f \shuffle^T \big( (B_+^x)^{\circ(m-i)} B_+^y(g) \big) \right] \\
	&\hspace{5cm} + \sum_{j=0}^{n} \binom{m-1+j}{j}  (B_+^x)^{\circ(m+j)}\circ B_+^y  \left[ \big( (B_+^x)^{\circ(n-j)}\circ B_+^y(f) \big) \shuffle^T g \right].
 \end{align*}
 Summing these two expressions and using Pascal's triangle (and the fact that $\binom{p-1}{0}=\binom{p}{0}$) gives the result, which concludes the induction and the proof.
\end{proof}
\begin{rk}
 This result can be understood in a purely combinatorial way. Each tree in this shuffle has to have at leat min$\{n,m\}$ $x$'s decorating its root and its descendants. These decorations can come either from the term $(B_+^x)^{\circ n}\circ B_+^y(f)$ or from the the term $(B_+^x)^{\circ m}\circ  B_+^y(g)$. The binomial coefficients come from all the possibilities one has to choose the $x$'s.
\end{rk}
We can now prove the main result of this Subsection, which gives an inductive formula to directly compute the $\yew$ product of two rooted forests, without references to the shuffle $\shuffle^T$ nor the binarisation map $\s^T$.
\begin{thm} \label{thm:yew_formula}
 The $\yew$ product admits the following inductive description:
 \begin{itemize}
	\item For any forest $F\in\calf_{\N^*}$,  $F \yew \emptyset = \emptyset \yew F = F$. 
	\item For any rooted trees $T_1 = B_+^n (f_1)$ and $T_2 = B_+^m(f_2)$:
	$$T_1 \yew T_2 = \sum_{i=0}^{m-1} \binom{n-1+i}{i} B_+^{n+i}\left[ f_1 \yew B_+^{m-i}(f_2) \right]  + \sum_{j=0}^{n-1} \binom{m-1+j}{j} B_+^{m+j} \left[ B_+^{n-j}(f_1) \yew f_2 \right].$$
	\item For any rooted forests $F_1 = T_1\cdots T_n$ and $F_2 = t_1\cdots t_k$:
	\[ F_1 \yew F_2 = \frac{1}{kn} \sum_{i=1}^k \sum_{j=1}^n \left(  (T_i \yew t_j) T_1 \cdots \widehat{T_i} \cdots T_k t_1 \cdots \widehat{t_j} \cdots t_n \right). \]
	\end{itemize}
\end{thm}
\begin{proof}
 \begin{itemize}
  \item The first point follows directly from the facts that $\s^T(\emptyset)=\emptyset$ and that $\emptyset$ is the unique neutral element for the shuffle product $\shuffle^T$.
  \item For the second point, let $T_1 = B_+^n (f_1)$ and $T_2 = B_+^m(f_2)$ be as in the Theorem. Then
  \begin{align*}
	& T_1 \yew T_2 = (\s^T)^{-1} ( \s^T(T_1) \shuffle^T \s^T(T_2))\qquad\text{by definition of }\yew  \\
    = & (\s^T)^{-1} \left( (B_+^x)^{\circ(n-1)}\circ B_+^y (\s^T(f_1)) \shuffle^T (B_+^x)^{\circ(n-1)}\circ B_+^y (\s^T(f_2)) \right) \qquad\text{by definition of }\s^T\\
    = & (\s^T)^{-1} \Bigg( \sum_{i=0}^{m-1} \binom{n-1+i}{i} (B_+^x)^{\circ(n-1+i)}\circ B_+^y \left[ \s^T(f_1) \shuffle^T (B_+^x)^{\circ(m-1-i)}\circ B_+^y (\s^T(f_2)) \right]  \\ 
		& + \sum_{j=0}^{n-1} \binom{m-1+j}{j} (B_+^x)^{\circ(m-1+j)}\circ B_+^y \left[ (B_+^x)^{\circ(n-1-j)}\circ B_+^y (\s^T(f_1)) \shuffle^T \s^T(f_2) \right] \Bigg)  ~ \text{by Lemma \ref{lem:yew_inter}} \\
    = & \sum_{i=0}^{m-1} \binom{n-1+i}{i} B_+^{n+i} \left( (\s^T)^{-1}\left[ \s^T(f_1) \shuffle^T \s^T(B_+^{m-i}f_2) \right] \right) \\ 
		& + \sum_{j=0}^{n-1} \binom{m-1+j}{j} B_+^{m+j} \left( (\s^T)^{-1} \left[\s^T( B_+^{n-j} f_1) \shuffle^T \s^T(f_2) \right] \right) \qquad\text{by definition of }(\s^T)^{-1} \\
    = & \sum_{i=0}^{m-1} \binom{n-1+i}{i} B_+^{n+i}\left[ f_1 \yew B_+^{m-i}(f_2) \right]  + \sum_{j=0}^{n-1} \binom{m-1+j}{j} B_+^{m+j} \left[ B_{n-j}(f_1) \yew f_2 \right]
	\end{align*}
	by definition of $\yew$.
  \item Finally, let $F_1 = T_1\cdots T_n$ and $F_2 = t_1\cdots t_k$ be as in the Theorem. Then 
  \begin{align*}
	& F_1 \yew F_2 = (\s^T)^{-1}( s^T(F_1) \shuffle^T \s^T(F_2)) \qquad\text{by definition of }\yew  \\
    = & (\s^T)^{-1} \Bigg( \frac{1}{kn} \sum_{i=1}^k \sum_{j=1}^n \Big(  (\s^T(T_i) \shuffle^T \s^T(t_j)) \s^T(T_1) \cdots \widehat{\s^T(T_i)} \cdots \s^T(T_k) \s^T(t_1) \cdots \widehat{\s^T(t_j)} \cdots \s^T(t_n) \Big) \Bigg) \\
    & \hspace{6cm}\llcorner\text{since }\s^T\text{ is an algebra morphism and by definition of }\shuffle^T \\
		&= \frac{1}{kn} \sum_{i=1}^k \sum_{j=1}^n \left(  (T_i \yew t_j) T_1 \cdots \widehat{T_i} \cdots T_k t_1 \cdots \widehat{t_j} \cdots t_n \right).
	\end{align*}
 \end{itemize}
\end{proof}
This result allows fast computations of $\yew$ products of rooted forests, espacially in cases with few branchings. Let us illustrate this with some examples.
\begin{ex} \label{ex:yew_products}
We start with an explicit computation:
 \begin{align*}
		\tdun{2} \yew \tddeux{3}{1} 
		&= B_+^2( \tddeux{3}{1}) + 2. B_+^3( \tddeux{2}{1}) + 3. B_+^4( \tddeux{1}{1}) + B_+^3( \tdun{2} \yew \tdun{1}) + 3. B_+^4( \tdun{1} \yew \tdun{1}) \\
		&= \tdtrois{2}{3}{1} + 2. \tdtrois{3}{2}{1} + 3. \tdtrois{4}{1}{1} + B_+^3 \left[ B_+^2(\tdun{1}) + B_+^1( \tdun{2}) + B_+^2( \tdun{1}) \right] + 3. B_+^4 \left[ B_+^1( \tdun{1}) + B_+^1( \tdun{1}) \right] \\
		&= \tdtrois{2}{3}{1} + 4. \tdtrois{3}{2}{1} + 9. \tdtrois{4}{1}{1} + \tdtrois{3}{1}{2}
	\end{align*}
 The same computation can be carried out with arbitrary coefficients. We obtain
 \begin{align*}
  \tdun{m} \yew \tddeux{n}{k} = & \sum_{i=0}^{m-1}\binom{n-1+i}{i}\left[\sum_{i'=0}^{k-1}\binom{m-i-1+i'}{i'}\tdtrois{n+i}{m-i+i'}{k-i'} + \sum_{j'=0}^{m-i-1}\binom{k-1+j'}{j'}\tdtrois{n+i}{k+j'}{m-i-j'}\right] 
+ \sum_{j=0}^{n-1}\binom{m-1+j}{j}\tdtrois{m+j}{n-j}{k}.
 \end{align*}
 We can carry further the computation. For example, one finds
 \begin{align*}
  \tddeux{n}{k}\yew\tddeux{m}{l} & =  \sum_{a=0}^{m-1}\binom{n-1+a}{a}\left\{ \sum_{i=0}^{k-1}\binom{m-a-1+i}{i}\left[\sum_{i'=0}^{l-1} \binom{k-i-1+i'}{i'}
  \tdquatre{n+a}{m-a+i}{k-i+i'}{l-i'} + \sum_{j'=0}^{k-i-1}\binom{l-1+j'}{j'}\tdquatre{n+a}{m-a+i}{l+j'}{k-i-j'}~
  \right]\right. \\
  + & \left.\sum_{j=0}^{m-a-1}\binom{k-1+j}{j}\tdquatre{n+a}{k+j}{m-a-j}{l}\right\} + \sum_{b=0}^{n-1}\binom{m-1+b}{b}\left\{\sum_{i=0}^{l-1}\binom{n-b-1+i}{i}\left[\sum_{i'=0}^{k-1}\binom{l-i-1+i'}{i'}\tdquatre{m+b}{n-b+i}{l-i+i'}{k-i'}    \right.\right.  \\
  & +  \left.\left. \sum_{j'=0}^{l-i-1}\binom{k-1+j'}{j'}\tdquatre{m+b}{n-b+i}{k+j'}{l-i-j'}~\right] + \sum_{j=0}^{n-b-1}\binom{l-1+j}{j}\tdquatre{m+b}{l+j}{n-b-j}{k}\right\}.
 \end{align*}

\end{ex}
 
 \subsection{The yew product and TZVs}
 
 In order to relate the $\yew$ product we just introduced with TZVs, we first need a simple property.
 \begin{prop}
  The $\yew$ product stabilises $\calf_{\N^*}^{\rm conv}$ i.e., for any $F_1,F_2\in\calf_{\N^*}^{\rm conv}$, $F_1\yew F_2\in \calf_{\N^*}^{\rm conv}$.
 \end{prop}
 \begin{proof}
  The result follows from the observation that $\s^T(\calf_{\N^*}^{\rm conv})=\calf_{\{x,y\}}^{\rm conv}$,  (\cite[Lemma A.3]{Cl20}), the fact that $\shuffle^T$ stabilises $\calf_{\{x,y\}}^{\rm conv}$ (\cite[Lemma 5.9]{Cl20}) and the obvious fact $\s^T$ is a one-to-one map.
 \end{proof}
 This ensures that the next simple but important result makes sense.
 \begin{thm} \label{thm:yew_TZVs}
  The map $\zeta^t:\calf_{\N^*}^{\rm conv}\longrightarrow\R$ is an algebra morphism for the $\yew$ product:
  \begin{equation*}
   \forall F_1,F_2 \in \calf_{\N^*}^{\rm conv}, \quad \zeta^t \left( F_1 \yew F_2 \right) = \zeta^t(F_1) \zeta^t(F_2).
  \end{equation*}
 \end{thm}
 \begin{proof}
  For any $F_1,F_2 \in \calf_{\N^*}^{\rm conv}$ we have
  \begin{align*}
		\zeta^t(F_1 \yew F_2) &= \zeta_\shuffle^T( \s( F_1 \yew F_2) ) &\text{by Theorem \ref{thm:integral_sum}}\\
			&= \zeta_\shuffle^T( \s(F_1) \shuffle^T \s(F_2)) 	&\text{by definition of }\yew \\
			&= \zeta_\shuffle^T(\s(F_1)) \zeta_\shuffle^T(\s(F_2))	&\text{since }\zeta_\shuffle^T\text{ is an algebra morphism for }\shuffle^T\\
			&= \zeta^t(F_1) \zeta^t(F_2)	&\text{by Theorem \ref{thm:integral_sum}}
	\end{align*}
 \end{proof}
 As for the shuffle product $\shuffle^T$, the non-associativity of the $\yew$ product implies the existence of relations amongst TVZs that have non direct equivalent for MZVs. These relations follow directly from the previous theorem and the associativity of the product on $\R$.
 \begin{coro}
  The image of the associator of $\yew$ restricted to $\calf_{\N^*}^{\rm conv}$ lies in $\ker(\zeta^t)$ i.e., for any $F_1,F_2$ and $F_3$ in $\calf_{\N^*}^{\rm conv}$ we have
  \begin{equation*}
   (F_1 \yew F_2) \yew F_3 - F_1 \yew (F_2 \yew F_3) \in \ker(\zeta^t).
  \end{equation*}
 \end{coro}
 We can now show that the $\yew$ product allows to relate $\zeta^t$ to the stuffle MZV map $\zeta_\stuffle$. We start by an important proposition which rely upon the assoicativity of the shuffle product of words $\shuffle$. In order to state this result, recall that a {\bf ladder tree} decorated by a set $\Omega$ is a rooted tree of the form $l=B_+^{\omega_1}\circ\cdots\circ B_+^{\omega_n}(\emptyset)$ for some $n\in\N^*$ and $(\omega_1,\cdots,\omega_n)\in\Omega^n$. A {\bf ladder forest} is a rooted forest $f=l_1\cdots l_n$ where the trees $l_i$s are all ladder trees.
 \begin{prop}
  The product $\yew$ is associative on ladder trees. In particular, for any ladder trees $l_1,\cdots,l_n$, the quantity $l_1\yew\cdots\yew l_n$ is well-defined and we furthermore have
  \begin{equation} \label{eq:yew_ladders_n}
   l_1 \yew \cdots \yew l_n = (\s^T)^{-1} ( \s^T(l_1) \shuffle^T \cdots \shuffle^T \s^T(l_n)).
  \end{equation}
 \end{prop}
 \begin{proof}
  For any forests $f$, $g$ and $h$, using the definition of $\yew$ we obtain
  \begin{equation*}
   (f\yew g)\yew h = (\s^T)^{-1}\left(\left(\s^T(f)\shuffle^T\s^T(g)\right)\shuffle^T\s^T(h)\right).
  \end{equation*}
  Furthermore, if $f$ and $g$ are ladder trees we have
  \begin{equation*}
   \s^T(f)\shuffle^T\s^T(g)=\iota\left(\s\left(\iota^{-1}(f)\right)\shuffle\s\left(\iota^{-1}(g)\right)\right)
  \end{equation*}
  with $\iota$ the canonical injection of words into rooted forests. The associativity of $\yew$ on ladders trees follows from this formula and the associativity of the usual shuffle product $\shuffle$ on words.
  
  For the second statement on the Proposition, we already have by the previous result that for any ladder trees $l_1,\cdots,l_n$, the quantity $l_1\yew\cdots\yew l_n$ is well-defined. Let us now prove Equation \eqref{eq:yew_ladders_n} by induction over $n\geq2$. 
  
  For $n=2$, Equation \eqref{eq:yew_ladders_n} is the definition of the yew product $\yew$. Assuming that for some $n\geq2$, the result holds for any $k\in\{2,\cdots,n\}$ take $l_1,\cdots,l_{n+1}$ be $n+1$ laddder trees. If $n+1=3$, the result holds from the computation we performed for the first point of the Proposition. If $n+1\geq4$ we have 
  \begin{align*}
	l_1 \yew \cdots\yew l_{n+1} &= ((\s^T)^{-1}(\s^T(l_1) \shuffle^T \s^T(l_2))) \yew l_3 \cdots \yew l_{n+1} \\
		&= ((\s^T)^{-1}(\s^T(l_1) \shuffle^T \s(l_2))) \yew (\s^T)^{-1}( \s^T(l_3) \shuffle^T \cdots \shuffle^T \s^T(l_{n+1})) \\
		&= (\s^T)^{-1} \left( \s^T(l_1) \shuffle^T \cdots \shuffle^T \s^T(l_{n+1})\right)
	\end{align*}
  by the associativity of $\yew$ on ladder trees, the induction hypothesis and the definition of $\yew$ respectively. This conclude our proof.
 \end{proof}
 This result allow us to define the following $\yew$-flattening map.
 \begin{defn} \label{defn:flattening_yew}
 The {\bf $\yew$-flattening map} $fl_\yew:\mathcal{F}_{\N^*}\longrightarrow\calw_{\N^*}$  from the algebra of rooted forests decorated by $\N^*$ and 
 the algebra of words $\mathcal{W}_{\N^*}$  written in the alphabet ${\N^*}$ is recursively defined by
 \begin{equation*}
  fl_\yew(\emptyset)=\emptyset,\quad fl_\yew(F_1F_2) = fl_\yew(F_1) \yew fl_\yew(F_2), \quad fl_\yew(B_+^n(F))=(n)\sqcup fl_\yew(F)
 \end{equation*}
 extended by linearity to a map on $\calf_{\N^*}$ and $\iota$ the canonical injection of words into rooted forests.
\end{defn}
Notice that we use the abuse of notations that we allowed ourrselves to make in Remark \ref{rk:yew_words}; namely that we use the $\yew$ product with words as arguments. Without this abuse of notations, the R.H.S. of the second term of the equation would be $\iota^{-1} \left( \iota \left( fl_\yew(F_1) \right) \yew \iota \left( fl_\yew(F_2) \right) \right)$.

We can now prove the main result of this Section.
\begin{thm} \label{thm:comm_fl_fl_yew}
 The flattening and binarisation maps are related by $fl_0 \circ \s^T = \s \circ fl_\yew$.
\end{thm}
\begin{proof}
 We will show that for any rooted forest $F\in\calf_{\N^*}$, we have $fl_0 \circ \s^T(F) = \s \circ fl_\yew(F)$ by induction on $n$, the number of vertices of $F$.
 
 If $n=0$, we have $F=\emptyset$ and $fl_0 \circ \s^T(\emptyset) = \emptyset = \s \circ fl_\yew(\emptyset)$ and the result hold for $n=0$. Now for some $n\in\N$ let us assume that the result holds for any rooted forest with $k\leq n$ vertices and let $F$ be a forest with $n+1$ vertices. We have two cases to consider:
 \begin{itemize}
  \item If $F$ is a rooted tree, then we can write $F=B_+^r(f)$ for some $r\in\N^*$ and rooted forests $f$ with $n$ vertices. Then we have
  \begin{align*}
		 \s \circ fl_\yew(F) &= \s \circ fl_\yew \circ B_+^r (f) \\
		 	&= \s \circ (r \sqcup fl_\yew (f)) &\text{by definition of $fl_\yew$}\\
&= (x\cdots xy)\sqcup (\s \circ fl_\yew)(f)  &\text{by definition of $\s$}\\
		 	&= (x\cdots xy)\sqcup ( fl_0 \circ \s^T)(f) &\text{by the induction hypothesis} \\
		 	&= fl_0\circ (B_+^x)^{\circ (r-1)} \circ B_+^y \circ \s^T(f) &\text{by definition of $fl_0$}\\
		 	&= fl_0 \circ \s^T \circ B_+^r (f) &\text{by definition of $\s^T$}\\
		 	&= fl_0 \circ \s^T(F).
  \end{align*}
  In this computation, $(x\cdots xy)$ was always $r-1$ $x$s and one $y$.
  \item If $F = t_1\cdots t_k$ with $t_1,\cdots, t_k$ rooted trees then:
		\begin{align*}
	  	 \s \circ fl_\yew(F) &= \s \circ( fl_\yew(t_1) \yew \cdots \yew fl_\yew(t_k)) &\text{by definition of $\yew$} \\
	  	 &= (\s \circ fl_\yew)(f_1) \shuffle \cdots \shuffle (\s \circ fl_\yew)(f_k) &\text{by Equation \eqref{eq:yew_ladders_n}} \\
	  	 &= (fl_0 \circ \s^T)(f_1) \shuffle \cdots \shuffle (fl_0 \circ \s^T)(f_k) &\text{by the induction hypothesis} \\
	  	 &= (fl_0 \circ \s^T) (f_1 \cdots f_k) &\text{by definition of $fl$}.
		\end{align*}
 \end{itemize}
 This conclude the induction and the proof.
\end{proof}
We need a simple Lemma to apply this result to TZVs.
\begin{lem}
 $fl_\yew$ maps convergent rooted forests to convergent words: $fl_\yew( \calf_{\N^*}^{\rm conv}) = \calw_{\N^*}^{\rm conv}$.
\end{lem}
\begin{proof}
 Notice that any convergent rooted tree $T$ is mapped by $fl_\yew$ to a linear combination of words starting by the decoration of the root of $T$. Furthermore, any convergent rooted forest $F$ is mapped by $fl_\yew$ to a linear combination of words starting by the decoration of the one of the root of $F$. Thus we have 
 $fl_\yew( \calf_{\N^*}^{\rm conv}) \subseteq \calw_{\N^*}^{\rm conv}$.
 
 Recall that $\iota:\calw_{\N^*}$ is the canonical injection of words into rooted forests, mapping any word to a ladder tree. Then from the definition of $fl_\yew$ we have, for any word $w\in\calw_{\N^*}$ (not necessarily convergent) $fl_\yew(\iota(w))=w$. This holds in particular for convergent words and we have therefore that $fl_\yew( \calf_{\N^*}^{\rm conv}) = \calw_{\N^*}^{\rm conv}$.
\end{proof}
We then obtain from Theorem \ref{thm:comm_fl_fl_yew} that the $\yew$-flattening relates TZVs and stuffle MZVs.
\begin{coro} \label{coro:TZVs_stuffle_MZVs}
 We have $\zeta^t = \zeta_\stuffle \circ fl_\yew$. In particular for any convergent forest $F\in\calf_{\N^*}^{\rm conv}$, $\zeta^t(F)$ is a linear combination with integer coefficients of stuffle MZVs given by the $\yew$-flattening. 
\end{coro}
\begin{proof}
 From the previous Lemma, we have that $\zeta^t$ and $\zeta_\stuffle \circ fl_\yew$ are defined on the same set. We need to prove that they are equal.
 \begin{align*}
	\zeta^t &= \zeta^T_\shuffle \circ \s^T &\text{by Proposition \ref{prop:crucial}}\\
	&= \zeta_\shuffle \circ fl_0 \circ \s^T &\text{by Theorem  \ref{thm:flattening}} \\
	&= \zeta_\shuffle \circ \s \circ fl_\yew  &\text{by Theorem \ref{thm:comm_fl_fl_yew}} \\
	&= \zeta_\stuffle \circ fl_\yew &\text{by Kontsevitch's relation \eqref{eq:Kontsevitch}}
	\end{align*} 
\end{proof}
In particular, we will use the $\yew$-flattening to compute some CZVs in Section \ref{section:five}. But first, we can use the computations performed in Example \ref{ex:yew_products} to express some TZVs in terms of MZVs.
\begin{ex}
 By definition of the $\yew$-flattening we have
 \begin{equation*}
  fl_\yew\left(\tdquatredeux{p}{m}{n}{k} \right) = (p)\sqcup((nk)\yew(m))
 \end{equation*}
 (with the abuse of notation discussed after Definition \ref{defn:flattening_yew}). Then by Corollary \ref{coro:TZVs_stuffle_MZVs} and the second computation of Example \ref{ex:yew_products} we obtain
 \begin{align*}
  & \zeta^t\left(\tdquatredeux{p}{m}{n}{k} \right) =  \sum_{i=0}^{m-1}\binom{n-1+i}{i}\left[\sum_{i'=0}^{k-1}\binom{m-i-1+i'}{i'}\zeta_\stuffle(p;n+i;m-i+i';k-i') \right. \\
  + & \left.\sum_{j'=0}^{m-i-1}\binom{k-1+j'}{j'}\zeta_\stuffle(p;n+i;k+j';m-i-j')\right]
+ \sum_{j=0}^{n-1}\binom{m-1+j}{j}\zeta_\stuffle(p;m+j;n-j;k).
 \end{align*}
 for any $p\geq2$, $(m,n,k)\in(\N^*)^3$.
 
 Finally, we can also use Corollary \ref{coro:TZVs_stuffle_MZVs} on the tree $\tcinq{p}{n}{k}{m}{l}$ and the third computation of Example \ref{ex:yew_products} to obtain
 \begin{align*}
  & \zeta^t\left(\tcinq{p}{n}{k}{m}{l}\right) =  \sum_{a=0}^{m-1}\binom{n-1+a}{a}\left\{ \sum_{i=0}^{k-1}\binom{m-a-1+i}{i}\left[\sum_{i'=0}^{l-1}\binom{k-i-1+i'}{i'}
  \zeta_\stuffle(p;n+a;m-a+i;k-i+i';l-i') \right]\right. \\
  + & \left. \sum_{j'=0}^{k-i-1}\binom{l-1+j'}{j'}\zeta_\stuffle(p;n+a;m-a+i;l+j';k-i-j')
   + \sum_{j=0}^{m-a-1}\binom{k-1+j}{j}\zeta_\stuffle(p;n+a;k+j;m-a-j;l)\right\} \\
 + &  \sum_{b=0}^{n-1}\binom{m-1+b}{b}\left\{\sum_{i=0}^{l-1}\binom{n-b-1+i}{i}\left[\sum_{i'=0}^{k-1}\binom{l-i-1+i'}{i'}\zeta_\stuffle(p;m+b;n-b+i;l-i+i';k-i')    \right.\right.  \\
  + & \left.\left. \sum_{j'=0}^{l-i-1}\binom{k-1+j'}{j'}\zeta_\stuffle(p;m+b;n-b+i;k+j';l-i-j')\right] + \sum_{j=0}^{n-b-1}\binom{l-1+j}{j}\zeta_\stuffle(p;m+b;l+j;n-b-j;k)\right\}.
 \end{align*}
 Notice that Corollary \ref{coro:TZVs_stuffle_MZVs}, together with Theorem \ref{thm:yew_formula} allows to (relatively) efficiently compute the evalutation of TZVs in terms of MZVs. In particular, it seems tractable for computers. The implementation of these results in formal language computation is left for future research. 
\end{ex}

\subsection{A new stuffle product}

With a similar strategy than the one used to define the $\yew$ product, we can define a product generalising the stuffle product for words.
\begin{defn} \label{defn:new_stuffle}
 Let $\stuffle^t:\calf_{\N^*}\otimes\calf_{\N^*}\longrightarrow\calf_{\N^*}$ be the product defined as the following composition:
 \begin{equation*}
  \stuffle^t:=\iota\circ\stuffle\circ(\s^{-1}\otimes\s^{-1})\circ(fl_0\otimes fl_0)\circ(\s^T\otimes\s^T)
 \end{equation*}
 where $\iota$ is the canonical map sending words to ladder trees. In other words, for any rooted forests $(F_1,F_2)\in(\calf_{\N^*})^2$ we have
 \begin{equation*}
  F_1\stuffle^t F_2:=\iota\left[\s^{-1}\left(fl_0\left(\s^T(F_1)\right)\right)\stuffle~\s^{-1}\left(fl_0\left(\s^T(F_1)\right)\right)\right].
 \end{equation*}
\end{defn}
Before stating the main properties of this new product, let us point out that its well-definedness is insured by the following diagram which is another way to define $\stuffle^t$.
\begin{equation*}
 \xymatrix{ \calf_{\N^*}\otimes\calf_{\N^*} \ar[r]^-{\s^T\otimes\s^T}  & \calf_{\{x,y\}}\otimes\calf_{\{x,y\}} \ar[r]^{fl_0\otimes fl_0} & \calw_{\{x,y\}}\otimes\calw_{\{x,y\}} \ar[r]^-{\s^{-1}\otimes\s^{-1}} & \calw_{\N^*}\otimes\calw_{\N^*} \ar[r]^-{\stuffle} & \calw_{\N^*} \ar[r]^-{\iota} & \calf_{\N^*}
 }
\end{equation*}
We now state all the useful properties of $\stuffle^t$ in only one Proposition.
\begin{prop} \label{prop:propriete_stuffle_t}
 $\stuffle^t$ is associative and commutative. The empty forest is the neutral element for this product. Furthermore, $\stuffle^t$ stabilises convergent forests.
\end{prop}
\begin{proof}
 Associativity follows from the associativity of the stuffle product $\stuffle$ and the fact that for any linear combination of words $w$ we have
 \begin{equation} \label{eq:trivial_stuffle_t}
  \s^{-1}\left(fl_0\left(\s^T(w)\right)\right)=w.
 \end{equation}
 Then, for any forests $F_1$, $F_2$ and $F_3$ in $\calf_{\N^*}$, writing $w_i:=\s^{-1}\left(fl_0\left(\s^T(F_i)\right)\right)$ we have
 \begin{align*}
  (F_1\stuffle^t F_2)\stuffle^t F_3 & = \left[\iota(w_1\stuffle w_2)\right]\stuffle^t F_3 \\
  & = \iota\left[\left(w_1\stuffle w_2\right)\stuffle w_3\right] &\text{ by Equation \eqref{eq:trivial_stuffle_t}} \\
  & = \iota\left[w_1\stuffle \left(w_2\stuffle w_3\right)\right] &\text{ by associativity of }\stuffle \\
  & = F_1\stuffle^t (F_2\stuffle^t F_3).
 \end{align*}
 Commutativity of $\stuffle^t$ follows directly from the commutativity of $\stuffle$.
 
 To see that the empty forest is the unit of $\stuffle^t$, recall that the empty word in $\calw_{\N^*}$ is the unit for the stuffle product $\stuffle$ in $\calw_{\N^*}$. The result then follows from the fact that $\s^{-1}\circ fl_0\circ\s^T$ sends the empty forest in $\calf_{\N^*}$ to the empty word in $\calw_{\N^*}$.
 
 Finally, the fact that convergent forests form a sub-algebra for the $\stuffle^t$ stuffle product simply follows from the fact that $\s^{-1}\circ fl_0\circ\s^T$ sends convergent forests into convergent words and that $\stuffle$ stabilises convergent words.


\end{proof}
The fact that $\stuffle^t$ stabilises convergent forests allow us to state the next result, which is the justification of the introduction of this product.
\begin{thm} \label{thm:TZVs_alg_morph}
 $\zeta^t$ is an algebra morphism for the $\stuffle^t$ product: for any convergent forests $(F_1,F_2)\in\calf_{\N^*}^{\rm conv}$ we have
 \begin{equation*}
  \zeta^t(F_1\stuffle^t F_2)=\zeta^t(F_1)\zeta^t(F_2).
 \end{equation*}
\end{thm}
 \begin{proof}
  This follows directly from the properties of the MZVs, AZVs and TZVs: for any forests $F_1$ and $F_2$ in $\calf_{\N^*}^{\rm conv}$, $F_1\stuffle^tF_2$ also belongs to $\calf_{\N^*}^{\rm conv}$ by Proposition \ref{prop:propriete_stuffle_t}. We then have
  \begin{align*}
   \zeta^t(F_1\stuffle^t F_2) & = \zeta^t\left(\iota\left[\s^{-1}\left(fl_0\left(\s^T(F_1)\right)\right)\stuffle~ \s^{-1}\left(fl_0\left(\s^T(F_2)\right)\right)\right]\right) \\
   & = \zeta_{\stuffle}\left(\s^{-1}\left(fl_0\left(\s^T(F_1)\right)\right)\stuffle~ \s^{-1}\left(fl_0\left(\s^T(F_2)\right)\right)\right) & \text{ by Equation \eqref{eq:TZVs_stuffle_words}} \\
   & = \zeta_{\stuffle}\left(\s^{-1}\left(fl_0\left(\s^T(F_1)\right)\right)\right)\zeta_{\stuffle}\left(\s^{-1}\left(fl_0\left(\s^T(F_2)\right)\right)\right) & \text{ by Equation \eqref{eq:stuffle_MZVs}} \\
   & = \zeta_{\shuffle}\left(fl_0\left(\s^T(F_1)\right)\right)\zeta_{\shuffle}\left(fl_0\left(\s^T(F_2)\right)\right)& \text{ by Kontsevitch's relation \eqref{eq:Kontsevitch}}\\
   & = \zeta_{\shuffle}^T\left(\s^T(F_1)\right))\zeta_{\shuffle}^T\left(\s^T(F_2)\right)& \text{ by Theorem \ref{thm:flattening}} \\
   & = \zeta^t(F_1)\zeta^t(F_2)
  \end{align*}
  by Proposition \ref{prop:crucial}. Thus $\zeta^t$ is an algebra morphism for $\stuffle^t$ as stated.
 \end{proof}
 Recall that the product $\stuffle^t$ is a generalisation of the usual stuffle product in the sense that we have $\iota(w_1\stuffle w_2)=\iota(w_1)\stuffle^t\iota(w_2)$. We therefore have a generalisation to rooted trees of the last property of MZVs (namely that MZVs are algebra morphisms for the stuffle product) that we set ourselves to generalise. This conclude our study of AZVs and TZVs for themselves and we will now turn our attention to consequences of these properties to other generalisation of MZVs.

 \section{Applications to Mordell-Tornheim zeta values} \label{section:three}

 A special class of Mordell-Tornheim zetas were introduced in \cite{To50} and studied (albeit not in full generality) in \cite{Mo58} and \cite{Ho92}. Later on they were further investigated in full generality in \cite{Ts05} and \cite{BrZh10}. The study of the finite version of Mordell-Tornheim zeta values was carried on in \cite{Ka16}. In this section we illustrate how Theorem \ref{thm:integral_sum}, together with the results of \cite{Cl20}, has far-reaching consequences for various generalisations of MZVs. In particular, it provides new and straightforward proofs for some results regarding Mordell-Tornheim zeta values.
 \begin{defn}
  Let $(s,s_1,\cdots,s_r)\in\N^{r+1}$ be sequence of non-negative integers. The {\bf Mordell-Tornheim zeta value} associated to this sequence is
  \begin{equation} \label{eq:MT_zeta}
   MT(s_1,\cdots,s_r|s) := \sum_{n_1,\cdots,n_r\geq1}\frac{1}{n_1^{s_1}\cdots n_r^{s_r}(n_1+\cdots+n_r)^s}
  \end{equation}
  whenever this series is convergent. The integer $r$ is called the {\bf depth} of this Mordell-Tornheim zeta values, $s+s_1+\cdots+s_r$ its {\bf weight} and $s_1+\cdots+s_r$ its {\bf partial depth}.
 \end{defn}
 Since the series in \eqref{eq:MT_zeta} is invariant under a permutation of the $s_i$, it is traditional to assume $s_1\leq\cdots\leq s_r$. We will follow this convention.
 
 Bradley and Zhou gave in \cite[Theorem 2.2]{BrZh10} a condition for the convergence of the series \eqref{eq:MT_zeta} to hold in the more general case where the $s_i$ are complex numbers. With our convention, this condition reads: if for any $k\in\{1,\cdots,r\}$, the inequality
  \begin{equation} \label{eq:conv_crit_BrZh}
   s+\sum_{i=1}^k s_i > k
  \end{equation}
  holds, then the series \eqref{eq:MT_zeta} converges.
  
  The same authors also proved in \cite[Theorem 1.1]{BrZh10} that any convergent Mordell-Tornheim zeta value of weight $w$ and depth $r$ can be written as a linear combination with rational coefficients of MZVs of weight $w$ and depth $r$. For finite Mordell-Tornheim zeta values, the same result was obtained in
  \cite[Theorem 1.2]{Ka16}.
  
  We will show here that our Theorem \ref{thm:tree_zeta_MZVs} gives elementary proofs of the results of Bradley and Zhou for the case $s_2>0$. We also provide an explicit formula for Mordell-Tornheim zeta values in the case $s_1=0$, which is reminiscent of \cite[Theorem 1.2]{Ka16} for finite Mordell-Tornheim zeta values.
  \begin{prop} \label{prop:MT_s1_0}
   The Mordell-Tornheim zeta values associated to the sequence $(s,s_1=0,s_2>0,\cdots,s_r)$ is convergent whenever $s\geq2$. In this case $MT(s_1=0,\cdots,s_r|s)$ can be written as a linear combination with integer coefficients of MZVs of weight $s+s_2+\cdots+s_r$ and depth $r$ given by 
   \begin{equation} \label{eq:MT_s1_0}
    MT(s_1=0,\cdots,s_r|s) = \zeta_\shuffle\left((\underbrace{x\cdots x}_{s-1}y)\sqcup \left((\underbrace{x\cdots x}_{s_2-1}y)\shuffle\cdots\shuffle(\underbrace{x\cdots x}_{s_r-1}y)\right)   \right).
   \end{equation}
  \end{prop}
  \begin{rk}
   Notice that the condition $s\geq2$ is equivalent to the convergence criterion \eqref{eq:conv_crit_BrZh} of Bradley and Zhou in the case $s_1=0$ and $s_2>0$.
  \end{rk}
  \begin{proof}
   Observe that in the case $s_1=0$ and $s_2>0$ the series \eqref{eq:MT_zeta} coincide with $\zeta^t(T)$ with $(T,d_T)$ the decorated tree with $r$ vertices and $r-1$ leaves. Its root is decorated by $s$ and its leaves by $s_2,\cdots,s_r$.
   Thus $T$ is a convergent tree whenever $s\geq2$. In this case we then have by Theorem \ref{thm:tree_zeta_MZVs}
   \begin{equation*}
    MT(s_1=0,\cdots,s_r|s) = (\zeta_\shuffle\circ fl_0 \circ\fraks^T)(T).
   \end{equation*}
   Now $\fraks^T(T)$ is a tree with only one branching vertex. The segment between the root and the branching vertex contains $s$ vertices, the first $s-1$ being decorated by $x$ and the last one by $y$. Each of the $r-1$ segments between the branching vertex and one leaf contain $s_i$ vertices (with $i\in\{2,\cdots,r\}$) whose first $s_i-1$ vertices are decorated by $x$ and the last one by $y$. Therefore its flattening is precisely
   \begin{equation*}
    fl(\fraks^T(T))=(\underbrace{x\cdots x}_{s-1}y)\sqcup \left((\underbrace{x\cdots x}_{s_2-1}y)\shuffle\cdots\shuffle(\underbrace{x\cdots x}_{s_r-1}y)\right).
   \end{equation*}
   This gives Equation \eqref{eq:MT_s1_0}. In particular $MT(s_1=0,\cdots,s_r|s)$ can be written as a linear combination with integer coefficients of MZVs of weight $s+s_2+\cdots+s_r$ and depth $r$ (the number of $y$s in each words appearing in the expression of $MT(s_1=0,\cdots,s_r|s)$).
  \end{proof}
  We are now ready to prove the main result of this section.
  \begin{thm} \label{thm:MT_tree}
   Let $s_1\in\N^*$. The Mordell-Tornheim zeta value associated with $(s,s_1,\cdots,s_r)$ is convergent whenever $s\geq1$. In this case $MT(s_1,\cdots,s_r|s)$ can be written as a linear combination of MZVs of weight $s+\sum_{i=1}^rs_i$ and depth $r$ with integer coefficients.
  \end{thm}
  \begin{rk}
   As before, the condition $s\geq1$ is equivalent to the convergence criterion \eqref{eq:conv_crit_BrZh} of Bradley and Zhou in the case $s_1>0$.
  \end{rk}
  \begin{proof}
   We prove this result by induction on the partial depth $n:=s_1+\cdots+s_r$ of $MT(s_1,\cdots,s_r|s)$. If $n=1$, the conditions $s_1\geq1$ and $s_i\geq s_1$ for all $i\in\{1,\cdots,r\}$ implies $r=1$ and $s_1=1$. In this case the Mordell-Tornheim zeta value reduce to the usual zeta value $\zeta(s+1)$ which is convergent if $s>0$.
   
   Now, assume the result holds for all Mordell-Tornheim zeta values of partial weight $n$ with $s_1\geq1$. Let $s$ and $s_1$ be greater or equal to one and $(s_1,\cdots,s_r)$ be an increasing sequence of integers such that $s_1+\cdots+ s_r=n+1$. Using the decomposition into simple fraction
   \begin{equation*}
    \frac{1}{n_1\cdots n_r(n_1+\cdots+n_r)} = \frac{1}{(n_1+\cdots+n_r)^2} \sum_{i=1}^r\prod_{\substack{j=1 \\j\neq i}}^r \frac{1}{n_j}
   \end{equation*}
   (which holds since $\prod_{\substack{j=1 \\j\neq i}}^r \frac{1}{n_j}=\frac{n_i}{n_1\cdots n_r}$) we obtain
   \begin{equation*}
    \frac{1}{n_1^{s_1}\cdots n_r^{s_r}(n_1+\cdots+n_r)^s} = \sum_{i=1}^r\frac{1}{n_1^{s_1}\cdots n_i^{s_i-1}\cdots n_r^{s_r}(n_1+\cdots+n_r)^{s+1}}.
   \end{equation*}
   Summing over the $n_i$ we obtain (up to an irrelevant permutation of the $s_i$)
   \begin{equation*}
    MT(s_1,\cdots,s_r|s) = \sum_{i=1}^r MT(s_1,\cdots,s_i-1,\cdots,s_r|s+1)
   \end{equation*}
   whenever the series of one of the two side converges. Examining each of the terms of the RHS we have then two cases to consider.
   \begin{enumerate}
    \item $s_i=1$. Then by Proposition \ref{prop:MT_s1_0} $MT(s_1,\cdots,s_i-1,\cdots,s_r|s+1)$ is convergent since $s+1\geq2$, and is a linear combination with integer coefficients of MZVs of depth $r$ and weight 
    \begin{equation*}
     s+1+\sum_{\substack{k=1\\k\neq i}}^r s_k.
    \end{equation*}
    \item $s_i\geq2$. In this case by the induction hypothesis $MT(s_1,\cdots,s_i-1,\cdots,s_r|s+1)$ is convergent and can be written as a linear combination of MZVs of weight $s+\sum_{i=1}^rs_i$ and depth $r$ with integer coefficients. 
   \end{enumerate}
   In any case the RHS of the equation above is convergent whenever $s\geq1$ and is then a finite sum of linear combinations of MZVs of weight $s+\sum_{i=1}^rs_i$ and depth $r$ with integer coefficients.
   
   This conclude the induction and proves the Theorem.
  \end{proof}
  Let us further point out that the decomposition formula
  \begin{equation} \label{eq:MT_decom}
    MT(s_1,\cdots,s_r|s) = \sum_{i=1}^r MT(s_1,\cdots,s_i-1,\cdots,s_r|s+1)
   \end{equation}
   which holds whenever $s$ and $s_1$ are both strictly positive might be of interest and was not, to the best of the author's knowledge, previously available in literature\footnote{however, a version for finite Mordell-Tornheim zeta values exists. It is the first equation of \cite[Corollary 5.2]{BTT21}. I am very thanksful to Masataka Ono for finding out this reference and kindly telling me about it.}. This formula, together with Proposition \ref{prop:MT_s1_0} allows us to derive expressions for the Mordell-Tornheim zeta values with $s_1>0$. Indeed, iterating \eqref{eq:MT_decom} until one of the $s_i$ is cancelled, we obtain 
   \begin{equation*}
    MT(s_1,\cdots,s_r|s) = \sum_{i=1}^r \sum_{p_1=0}^{s_1-1}\cdots\sum_{p_{i-1}=0}^{s_{i-1}-1}\sum_{p_{i+1}=0}^{s_{i+1}-1}\cdots\sum_{p_r=0}^{s_r-1} MT\left(q_1,\cdots,q_r|s+\sum_{j=1}^r p_j\right)
   \end{equation*}
   with in each of the terms in the RHS, $p_i:=s_i$ and $q_j:=s_j-p_j$ for $q\in\{1,\cdots,r\}$. Each of the Mordell-Tornheim zeta values in the RHS are of the type treated by Proposition \ref{prop:MT_s1_0} since $q_i=0$ for each $i$ in the outermost sum. Thus we obtain
   \begin{align} \label{eq:expression_MT}
    & MT(s_1,\cdots,s_r|s)  =  \\
    \nonumber 
\sum_{i=1}^r \sum_{p_1=0}^{s_1-1}\cdots & \sum_{p_{i-1}=0}^{s_{i-1}-1}\sum_{p_{i+1}=0}^{s_{i+1}-1}\cdots\sum_{p_r=0}^{s_r-1} \zeta_\shuffle\left(\fraks\left(s+\sum_{\substack{j=1}}^r p_j\right)\sqcup \Big(\fraks(s_1-p_1)\shuffle\cdots\shuffle\fraks(s_{r}-p_{r})\Big)   \right)
   \end{align}
   where in each of the terms in the RHS we have set $p_i:=s_i$ and used the convention $\fraks(0):=\emptyset$.

\section{Applications to conical zeta values} \label{section:four}

\subsection{Conical zeta values}

Let us start by recalling some classical definition of cones \cite{Fulton93,Ziegler94}
\begin{defn} \label{defn:cones}
\begin{itemize}
 \item Let $v_1,\cdots,v_n$ by $n$ linearly independant nonzero vectors in $\Z^k$. The {\bf cone} associated to these vectors is
 \begin{equation*}
  C=\langle v_1,\cdots,v_n\rangle := \R^*_+v_1+\cdots\R^*_+v_n.
 \end{equation*}
 If furthermore $k=n$, the cone is called {\bf maximal}. We write $\calc$ the set of maximal cones. By convention, the empty set is a maximal cone.
 \item A {\bf decorated cone} is a pair $(C,\vec s)$ with $C=\langle v_1,\cdots,v_n\rangle$ a cone and $\vec s\in\N^n$. We call $\calc_\N$ the set of decorated maximal cones.
 \item  For a cone $C$ (resp. a decorated cone $(C,\vec s)$), write the vectors $v_i$s in the canonical basis $\mathcal{B}=\{e_1,\cdots,e_n\}$ of $\R^n$: $v_i=\sum_{j=1}^na_{ij}e_j$. Then $A_C:=(a_{ij})_{i,j=1,..,n}$ is the {\bf representing matrix} (in the basis $\mathcal B$) of the cone $C$ (resp. $(C,\vec s)$).
 \item A cone $C$ (resp. a decorated cone $(C,\vec s)$) is {\bf unimodular} if its representing matrix $A_C$ has only $0$s and $1$s in its entries.
\end{itemize}
\end{defn}
Unless specified otherwise, the cones considered in this paper will be maximal, therefore we write cones instead of maximal cones. 
\begin{rk}
 The definition above only covers open simplicial rational smooth cones. More general cones, or closed ones, will play no role here thus we do not introduce them. Similarly, in the case of decorated cones, one could have $\vec s\in\C^n$. Notice that a cone is invariant under permutations of the $v_i$s, while a decorated cone in general is only invariant under the simultaneous permutations of the $v_i$s and the components of $\vec s$.
\end{rk}
Conical zeta values (CZVs) were introduced in \cite{GPZ13} as  weighted sums on integers points on cones. A generalisation of these objects was introduced before in \cite{Te04}. In \cite{Ze17} a description of CZVs in terms of matrices was given. We adopt here an intermediate definition, where the cone is encoded by a matrix but not the weight. Our definition is rigorously equivalent to the ones in \cite{GPZ13,Te04,Ze17} but is more suitable for our purpose.
\begin{defn} \label{defn:useful_stuff_cone}
 \begin{itemize}
  \item Let $\mathcal{B}=\{e_1,\cdots,e_n\}$ be the canonical basis of $\R^n$ and $l:\R^n\mapsto\R$ be a linear map defined by $l(v)=\sum_{i=1}^n l_iv_i$, with $v_i$ the coordinates of the vector $v$ in the canonical base: $v=\sum_{i=1}^n v_ie_i$. Then we say that $l$ is (resp. strictly) {\bf positive (with respect to $\mathcal{B}$)}, and we write $l\geq0$\footnote{the more rigorous notation $l\geq_\mathcal{B}0$ is not necessary since we always work with the canonical basis of $\R^n$.} (resp $l>0$) if $l_i\geq0$ (resp. $l>0$) for any $i\in[n]$.
  
  \item Let $C=\langle v_1,\cdots,v_n\rangle$ (resp. $(C=\langle v_1,\cdots,v_n\rangle,\vec{s}$)) be a maximal cone (resp. decorated cone). Write the vectors $v_i$s in the canonical basis $\mathcal{B}$: $v_i=\sum_{j=1}^na_{ij}e_j$ and set $l_i:(\R)^n\longrightarrow\R$ the linear maps defined by $l_i(w)=\sum_{j=1}^na_{ij}w_j$. Then  
  $C$ (resp. $(C,\vec s)$) is (resp. strictly) {\bf positive (with respect to $\mathcal{B}$)} if $l_i\geq0$ (resp. $l_i>0$) for any $i\in[n]$.
  
  \item Let $(C=\langle v_1,\cdots,v_n\rangle,\vec{s}=(s_1,\cdots,s_n))$ be a decorated (maximal) strictly positive cone and $l_i:(\R)^n\longrightarrow\R$ the associated linear maps as before. Then the {\bf conical zeta value} associated to $(C,\vec s)$  is
 \begin{equation} \label{eq:def_CZV}
  \zeta(C,\vec s) := \sum_{\vec{n}\in(\N^*)^n}\frac{1}{l_1(\vec{n})^{s_1}\cdots l_n(\vec{n})^{s_n}}
 \end{equation}
 whenever the series converge.
 
 We also denote by $\zeta$ the map which to a cone $(C,\vec s)$ associates the CZV $\zeta(C,\vec s)$ when it exists.
 \end{itemize}
\end{defn}
\begin{rk}
 Notice that any unimodular cone is strictly positive with respect to the canonical basis. This justifies that we will not require our cones to be strictly positive since they will be unimodular.
\end{rk}
There are many important open questions concerning CZVs. 
An important one is the existence of linear relations with rational coefficients between CZVs. It was shown in \cite{GPZ13} that they obey a family of relations given by double subdivisions of cones which conjecturally generate all linear relations with rational coefficients between by CZVs. Another question is the number-theoretic content of CZVs. It was shown in \cite{Te04} that CZVs are evaluation of polylogarithms at $N$-th roots of unity. A conjecture by Dupont, refined by Panzer (and written in \cite{Ze17}, Conjecture 2\footnote{therefore, I suggest to call this conjecture the Dupont-Panzer-Zerbini conjecture, or DPB for short.}) states that for a cone $C$, $N$ is the least common multiplier of the minors of $A_C$. In this paper, we answer this second question in the case of tree-like cones.

\subsection{From trees to cones}

It is clear that tree zeta values are conical zeta values. Let us start by making this statement rigorous.

Let $F\in\calf$ be a  rooted forest with vertices $V(F)=\{1,\cdots,N\}$. We write $A_F=(a_{ij})_{i,j=1,\cdots,N}$ the $N\times N$ matrix\footnote{this matrix is sometimes called the path matrix of $F$.} defined by
\begin{align*}
 a_{ij} = \begin{cases}
           & 1 \quad\text{if }i\preceq j\\
           & 0 \quad\text{otherwise}.
          \end{cases}
\end{align*}
Furthermore, let us set 
\begin{equation*}
 v_i(F):=\sum_{j=1}^Na_{ij}e_j
\end{equation*}
(with $\{e_1,\cdots,e_N\}$ the canonical basis of $\R^N$).
\begin{lem}
 The above construction, extended by linearity, defines a map from rooted forest to cones.
\begin{align*}
 \Phi:\calf~&\longrightarrow~ \calc \\
 F~&\longmapsto \langle v_1(F),\cdots,v_N(F)\rangle.
\end{align*}
\end{lem}
\begin{proof}
 Recall that all our cones are maximal cones. 
 We need to prove that, for any $F\in\calf$, the vectors $v_i(F)$ are linearly independant. We prove this by induction on $N=|V(F)|$. If $N=0$, then $F=\emptyset$ and $\Phi(\emptyset)=\emptyset$ is a cone. The case $N=1$ also trivially holds.
 
 Assume the result holds for all forests with $k\leq N$ vertices and let $F$ be a forest with $N+1$ leaves. Without loss of generality we can assume $V(F)=[N+1]$. We have two cases to consider: First, if $F=F_1F_2$ with $F_1$ and $F_2$ non empty, the result holds by the induction hypothesis used on $F_1$ and $F_2$ and the fact that the $v_i(F_1)$ and the $v_j(F_2)$ belong into two orthogonal subspaces of $\R^{N+1}$. 
 
 Second, if $F=T=B_+(\tilde F)$, we can assume without loss of generality that $N+1$ is the root of $T$. Then we have $v_{N+1}(T)=\sum_{k=1}^{N+1}e_k$ and $v_i(T)=v_i(\tilde F)$ for any $i\in[N]$. Thus 
 \begin{equation*}
  \sum_{k=1}^{N+1}\lambda_k v_k(T) = 0 ~\Longleftrightarrow~\lambda_{N+1}=0~\wedge~\sum_{k=1}^{N}\lambda_k v_k(\tilde F)=0
 \end{equation*}
 since $v_{N+1}(T)$ is the only vector with a non-zero $e_{N+1}$ component. The result then holds from the induction hypothesis used on $\tilde F$, which conclude the proof.
\end{proof}
\begin{rk}
 There is of course other ways to maps rooted forests to cones. We simply present the one that respects the conical and arborified zeta maps, as we explain below.
\end{rk}
The map $\Phi$ can be lifted to a map $\overline{\Phi}$ acting on decorated forests. 
 Let $\overline{\Phi}:\calf_{\N}\mapsto\calc_\N$ be the map defined by
 $\overline{\Phi}(F,d_F):=(\Phi(F),\vec{s}_F)$
 for any decorated forest $(F,d_F)\in\calf_{\N}$ with vertices $V(F)=\{1,\cdots,N\}$ and where we have set $\vec s_F:=(d_F(1),\cdots,d_F(N))$.
 
 This maps $\Phi$ and $\overline{\Phi}$  are not surjective. This justifies the following definition:
\begin{defn}
 A cone (resp. a decorated cone) is said to be a {\bf tree-like cone} (resp. a {\bf decorated tree-like cone}) when it lies in the image of $\Phi:\calf\longrightarrow\calc$ (resp. $\overline{\Phi}:\calf_{\N}\mapsto\calc_\N$). We write $\calct$ the set of tree-like cones.
 
 Furthermore, if a decorated cone lies in $\overline{\Phi}(\calf_{\N^*}^{\rm conv})$ it is called a {\bf convergent decorated tree-like cone}. We write $\calct_\N$ the set of decorated tree-like cone and $\calct_\N^{\rm conv}$ the set of convergent decorated tree-like cone.
\end{defn}
Let us recall before (Definition \ref{def:conv_forest_series}) that a $\N^*$-decorated forest is convergent if the decorations of each of its roots are greater or equal to 2.
The following simple properties of the map $\overline{\Phi}$ are a key result.
\begin{prop} \label{prop:forest_cones}
 For any non-empty convergent forest $(F,d_F)\in\calf_{\N^*}^{\rm conv}$, the conical zeta values $\zeta(\overline{\Phi}(F,d_F))$ is convergent and 
 \begin{equation*}
  \zeta(\overline{\Phi}(F,d_F)) = \zeta^t(F,d_F).
 \end{equation*}
\end{prop}
\begin{proof}
 Let $(F,d_F)$ be any convergent forest. Up to relabelling, we can identify its vertices set $V(F)$ with $[N]$ for some $N\in\N$: $V(F)=\{1,\cdots,N\}$. First let us  observe that, for any $\vec n=(n_1,\cdots,n_N)\in\N^N$ and $i\in V(F)$ we have 
 \begin{equation*}
  l_i(F)(\vec n) = \sum_{\substack{j\in[N]\\j\succeq i}}n_j
 \end{equation*}
 by definition of $l_i(F)$, and where $\succeq$ is the partial order of the set vertices $V(F)$ of $(F,d_F)$. Then Equation \eqref{eq:TZV} applied to the convergent rooted forest $(F,d_F)$ gives
 \begin{equation*}
  \zeta^t(F,d_F) = \sum_{\vec n\in\N^N}\prod_{i=1}^N\left(\sum_{\substack{j\in[N]\\j\succeq i}}n_j\right)^{-d_F(i)} = \sum_{\vec n\in\N^N} \frac{1}{l_1(F)(\vec n)^{d_F(1)}\cdots l_N(F)(\vec n)^{d_F(N)}} = \zeta(\overline{\Phi}(F),\vec s_F)
 \end{equation*}
 as claimed in the Proposition. The convergence of $\zeta(\overline{\Phi}(F),\vec s_F)$ then follows from the first point of Proposition \ref{prop:crucial}.
\end{proof}
\begin{rk}
 This implies in particular that shuffle AZVs are Shintani zetas. As such, if we take the decorations of $F$ to be complex parameters, the function $\vec s\mapsto\zeta^t(F)$ admits a meromorphic continuation to $\C^{|V(F)|}$ \cite{Ma03,Lo21}. The questions of analytic continuation of tree zeta values and their renormalisation lie well beyond the scope of this article, nonetheless Proposition \ref{prop:forest_cones} answers these questions, at least partially.
\end{rk}
Proposition \ref{prop:forest_cones} allows to effortlessly show the following important result.
\begin{thm} \label{thm:tree_CZVs_MZVs}
 For any convergent decorated tree-like cone $(C,\vec s)=\overline{\Phi}(F,d_F)$, the associated conical zeta values $\zeta(C,\vec s)$ is a linear combinations of MZVs with rational coefficients given by
 \begin{equation*}
  \zeta(C,\vec s) = (\zeta_\shuffle\circ fl\circ\fraks^T)(F,d_F).
 \end{equation*}
\end{thm}
\begin{proof}
 For any convergent decorated cone $(C,\vec s)=\overline{\Phi}(F,d_F)$ the result follows from Proposition \ref{prop:forest_cones} and Theorem \ref{thm:tree_zeta_MZVs} applied to $\zeta^t(F,d_F)$.
\end{proof}
Theorems \ref{thm:tree_zeta_MZVs} and \ref{thm:tree_CZVs_MZVs}, together with Proposition \ref{prop:forest_cones} and the results of \cite{Cl20} can be summarised as the commutativity of Figure \ref{fig:MZV_AZV} (where the CZVs, AZVs and MZVs maps are all written $\zeta$).
\begin{figure}[h!] 
  		\begin{center}
  			\begin{tikzpicture}[->,>=stealth',shorten >=1pt,auto,node distance=3cm,thick]
  			\tikzstyle{arrow}=[->]
  			
  			\node (1) {$\calf_{\N^*}^{\rm conv}$};
  			\node (2) [right of=1] {$\calf_{\{x,y\}}^{\rm conv}$};
  			\node (3) [right of=2] {$\calw_{\{x,y\}}^{\rm conv}$};
  			\node (4) [below of=1] {$\calct_{\N}^{\rm conv}$};
  			\node (5) [right of=4] {$\R$};

  			\path
  			(1) edge node [above] {$\fraks^T$} (2)
  			(2) edge node [above] {$fl$} (3)
  			(1) edge node [left] {$\overline{\Phi}$} (4)
  			(4) edge node [above] {$\zeta$} (5)
  			(1) edge node [above right] {$\zeta^t$} (5)
  			(2) edge node [right] {$\zeta_\shuffle^T$} (5)
  			(3) edge node [below right] {$\zeta_\shuffle$} (5);
  			\end{tikzpicture}
  			\caption{CZVs, AZVs, MZVs and tree zeta values.}\label{fig:MZV_AZV}
  		\end{center}
  	\end{figure} \\
Theorem \ref{thm:tree_CZVs_MZVs} answer one of the important question about CZVs for the convergent tree-like cones. Therefore it is useful to be able to characterise which cones are (convergent) tree-like. The next Subsection is dedicated to this question.
\begin{rk}
 In \cite[Lemma 1]{Ze16}, another sufficient condition was given for a unimodular CVZs to be a linear combinations of MZVs with rational coefficients. This condition, called C1 in \cite{Ze17}, the rows of the representing matrix $A_C$ can be permuted such that the 1s of $A_C$ are consecutive in each column. It is easy to see with counter-example that this condition has no link with being tree-like.
 
 The dual condition, let us call it C2, is that the columns of $A_C$ can be permuted so that the 1s of $A_C$ are consecutive in each column. To the best of the author's knowledge, it is still a conjecture (proposed by F. Zerbini in \cite{Ze17}) that C2 is another sufficient (but not necessary) condition for unimodular CZVs to be a linear combinations of MZVs with rational coefficients. It is easy to show that being tree-like implies the C2 conditon, but that the converse does not hold. Thus, Theorem \ref{thm:tree_CZVs_MZVs} is an indication that Zerbini's conjecture holds.
\end{rk}

\subsection{Characterisation of tree-like cones}

For conical zeta values, each lines of the representing matrix gives one term of the denominator of Equation \eqref{eq:def_CZV}. In Equation \eqref{eq:TZV} such a term corresponds to a vertex. Notice that, in this Equation, if $v_1\preceq v_2$, the term in the series associated to $v_1$ is
\begin{equation*}
 l_{v_1}(\vec n) = \sum_{\substack{v'\succeq v_1\\v'\not\succeq v_2}}n_{v'} + \sum_{v'\succeq v_2}n_{v'}
\end{equation*}
and that the second term in the RHS is the term in Equation \eqref{eq:TZV} associated to $v_2$. This justifies the following definition.
\begin{defn}
 For any $n$-dimensional cone $C$ (resp. decorated cone $(C,\vec s)$), let $\preceq_C$ be the relation on $[n]$ defined by
 \begin{equation*}
  i\preceq_C j ~:\Longleftrightarrow~ l_i-l_j\geq0
 \end{equation*}
 with the linear maps $l_i$ and the notion of positive linear maps of Definition \ref{defn:useful_stuff_cone}. As before, we write $\succeq_C$ the inverse relation.
\end{defn}
\begin{lem} \label{lem:trivial_poset}
 For any $n$-dimensional cone $C$ (resp. decorated cone $(C,\vec s)$), $([n],\preceq_C)$ is a poset.
\end{lem}
\begin{proof}
 Reflexivity and transitivity are trivial. Anti-symmetry follows from the fact that $C$ is a maximal cone, so two different lines of the representing matrix $A_C$ have to be different. Thus $l_i=l_j$ implies $i=j$ for any $i$ and $j$ in $[n]$.
\end{proof}
\begin{rk} \label{rk:connected_components}
 Using permutations of lines and columns, one can assume that the representing matrix is block-diagonal: $A_C={\rm diag}(A_1^C,\cdots,A_p^C)$ with $p\geq1$. Then the poset $([n],\preceq_C)$ has $p$ connected components.
\end{rk}
Thus we have defined a map 
\begin{align} \label{defn_Psi}
 \Psi:\mathcal{C} & \longrightarrow \mathcal{P}^{\rm fin} \\
 C & \longmapsto ([n],\preceq_C) \nonumber
\end{align}
where $\mathcal{P}^{\rm fin}$ is the set of finite posets. We can lift $\Psi$ to a map $\overline\Psi :\calc_\N\longmapsto\mathcal{P}^{\rm fin}_\N$ from decorated cones to decorated posets. Set $\overline\Psi(C,\vec s):=(\Psi(C),d_C)$, with $d_C:[n]\mapsto\N$ defined by $d_C(i):=s_i$.

Since not every cone is a tree-like cone, and more generally, not every conical sums is indexed by a poset, we need a compatibility condition on the cone to ensure that its associated conical sum respect the poset structure associated to the cone. One finds out the right condition by observing that in \eqref{eq:TZV}, if $v'\succeq v$, then the term coming from $v$ contains $n_{v'}$. 
\begin{defn}
 A $n$-dimensional cone $C$ (resp. decorated cone $(C,\vec s)$) with representing matrix $A_C=(a_{ij})_{i,j=1,\cdots,n}$ is 
 {\bf poset compatible} if, for any $i,j\in[n]$ we have
 \begin{equation*}
  a_{ij} \neq0~\Longrightarrow~i\preceq_C j.
 \end{equation*}
\end{defn}
So any $n$-dimensional poset compatible cone gives a conical zeta value whose sum is given by topological ordering of the poset $([n],\preceq_C)$. To check that this poset is a rooted forest, we need to introduce one more object. 
\begin{defn}
 For any $n$-dimensional cone $C$ (resp. decorated cone $(C,\vec s)$), its {\bf second representing matrix} $B_C:=(b_{ij})_{i,j=1,..,n}$ is the incidence matrix of the Hasse diagram of $([n],\preceq_C)$.
 
 In other words, $b_{ij}=1$ when $j$ is a direct successor of $i$ and $0$ otherwise:
 \begin{equation*}
  b_{ij}=1~\Longleftrightarrow~i\preceq_C j~\wedge~(\forall k\in[n],~i\preceq_Ck\preceq_Cj\Rightarrow k\in\{i,j\}).
 \end{equation*}
\end{defn}
We then have 
\begin{lem} \label{lem:carac_tree_cone}
 For a $n$-dimensional cone $C$ (resp. decorated cone $(C,\vec s)$), the poset $\Psi(C)=([n],\preceq_C)$ is a rooted forest if, and only if, its second representing matrix has at most one non-zero component per column.
\end{lem}
\begin{proof}
 An oriented graph is a rooted forest if, and only if it has no oriented cycle, no non-oriented cycle and each of its connected components has exactly one minimal element.
 
 The first point is guarantied by Lemma \ref{lem:trivial_poset}. The second and third points are equivalent to asking that each vertex has at most one direct ancestor for the relation $\preceq_C$. Since $j$ is a direct ancestor of $i$ if, and only if, $b_{ij}=1$, we have that $j$ has at most one direct ancestor if, and only if, it exists at most one $i\in[n]$ such that $b_{ij}=1$, thus that the $j$-th column of $B_C$ has at most one non-zero entry. Since this must hold for all $j\in[n]$, we have the result
\end{proof}
\begin{rk}
 For a cone $C$ as in Lemma \ref{lem:carac_tree_cone}, the forest $\Psi(C)$ has $p$ connected components (i.e. trees) if, and only if $A_C$, the representing matrix of $C$, is a block-diagonal matrix with exactly $p$ blocks (modulo permutations of lines and columns).
\end{rk}
Lemma \ref{lem:carac_tree_cone} actually gives a characterisation of tree-like cones. More precisely we have
\begin{prop} \label{prop:cone_to_tree}
 A unimodular cone $C$ (resp. a decorated unimodular cone $(C,\vec s)$) is a tree-like cone if, and only if, it is poset compatible and its second representing matrix has at most one non-zero entry in each column. Furthermore, the map $\Psi$ (resp. $\overline{\Psi}$) restricted to $\calct$ (resp. $\calct_\N$) is the inverse of the map $\Phi$ (resp. $\overline{\Phi}$).
\end{prop}
\begin{proof}
 The maps $\Phi$ and $\overline{\Phi}$ are clearly injective since $F$ is an isomorphism class of rooted forests.

 Let $C$ be any tree-like cone and $F$ its preimage under $\Phi$. Since $F$ is an isomorphism class, we can assume without loss of generality that $V(F)=[n]$. Then by construction of $\Phi$ and $\preceq_C$, the later is the partial order relation on $[n]=V(F)$ that defines the forest structure of $F$. Thus $B_C$ is the adjacency matrix of $F$ and since $\Psi(C)$ is by definition the poset whose adjacency matrix is $B_C$ we obtain $\Psi|_{\calct}=\Phi^{-1}$ as claimed.
 
 We have also shown that, if $C$ is a tree-like cone, then $\Psi(C)$ is a rooted forest. Thus, by Lemma \ref{lem:carac_tree_cone}, $C$ is poset compatible and $B_C$ has at most one non-zero entry in each column. We thus have that each tree-like cone satisfies the hypothesis of Lemma \ref{lem:carac_tree_cone}.
 
 The sole thing left to be shown is the fact that, if a cone obeys the hypothesis of Lemma \ref{lem:carac_tree_cone}, namely that its second representing matrix has at most one non-zero component per column, then it is a tree-like cone. By Lemma \ref{lem:carac_tree_cone} for such a cone $C$, $\Psi(C)\in\calf$. Using the same argument as in the first point of this proof, we have that $\Phi(\Psi(C))=C$. Thus $C\in\calct$ by definition of tree-like cones.
 
 The same results hold with the lifted maps $\overline\Phi$ and $\overline\Psi$ by definition of the lifts.
\end{proof}
Summing up the results of the last two sections in one statement we obtain
\begin{thm} \label{thm:main_result}
 Let $\calc=(C,\vec s)$ be a decorated unimodular cone such that its second representing matrix has at most one non-zero component per column. Then, provided that $\overline\Psi(\calc)$ is a convergent forest, $\zeta(C,\vec s)$ converges and is a linear combination of MZVs of weights $||\vec s||:=s_1\cdots+s_n$ with rational coefficients. They evaluate as:
 \begin{equation*}
  \zeta(C,\vec s) = (\zeta_\shuffle\circ fl_0\circ \fraks^T\circ\overline\Psi)(C,\vec s).
 \end{equation*}
\end{thm}
\begin{proof}
 This Theorem is a reformulation of Theorem \ref{thm:tree_CZVs_MZVs} together with the results of Lemma \ref{lem:carac_tree_cone} and Proposition \ref{prop:cone_to_tree}.
\end{proof}

\section{Computations of CZVs} \label{section:five}

The results of the previous subsection give us an algorithm to compute some CZVs. The steps to follow to compute $\zeta(C,\vec s)$ are first to check that $C$ is poset compatible, second to compute the second representing matrix of $C$ and to check that it has at most one $1$ per column. The third step is applying the branched binarisation map $\fraks^T$ to $\overline{\Psi}(C,\vec s)$.
 Then one has to flatten the obtained rooted forest with the map $fl_0$.
 Finally applying the MZVs map $\zeta_\shuffle$ gives the result. In some case, it is simpler to directly apply the $\yew$-flattening with Theorem \ref{thm:yew_formula} to make the computation.
 
We illustrate this procedure by computing some CZVs. Before this, let us point out that this algorithm produces the same result than the procedure of \cite[Definition 3.16]{OSY21}. Indeed, this procedure is what we called flattening (Definition \ref{defn:flattening}) but built from leaves rather than from the root (and excluding forests). Furthermore the harvestable trees of \cite{OSY21} are essentially our convergent trees in $\calf_{\{x,y\}}$. Then \cite[Theorem 3.17]{OSY21} is the finite version of \cite[Theorem 4.15]{Cl20}, which we use intensively for our computation.
\begin{ex} \label{ex:simple_cone}
 The simplest non-trivial example is:
 \begin{equation*}
  \zeta(C_1,\vec s_1):=\sum_{p,q,r\geq1}\frac{1}{(p+q+r)^2qr}.
 \end{equation*}
 We then have the representing and second representing matrices to be respectively:
 \begin{equation*}
  A_{C_1}= \begin{pmatrix}
            1 & 1 & 1 \\
            0 & 1 & 0 \\
            0 & 0 & 1
           \end{pmatrix},\quad 
  B_{C_1}= \begin{pmatrix}
            0 & 1 & 1 \\
            0 & 0 & 0 \\
            0 & 0 & 0
           \end{pmatrix}.
 \end{equation*}
 One easily check from $A_{C_1}$ that $C$ is poset compatible and clearly $B_C$ satisfies the hypothesis of Lemma \ref{lem:carac_tree_cone}. We further have
 \begin{equation*}
  \overline\Psi(C_1,(2,1,1)) = \tdtroisun{$2$}{$~1$}{$1$}
 \end{equation*}
 which is convergent. Applying the branched binarisation map $\fraks^T$ and the MZVs map $\zeta_\shuffle$ gives the result
 \begin{equation*}
  \sum_{p,q,r\geq1}\frac{1}{(p+q+r)^2qr} = 2\zeta_\stuffle(2,1,1).
 \end{equation*}
\end{ex}
The exact same computation can be used, up to the flattening, for higher powers in the denominator. We readily obtain:
\begin{prop}
 For any $\vec s=(n,m,l)\in\N^3$ with $n\geq2$ and $m,l\geq1$ we have that the CZVs
 \begin{equation*}
  \zeta(C_1,\vec s)=\sum_{p,q,r\geq1}\frac{1}{(p+q+r)^nq^mr^l}
 \end{equation*}
  are linear combinations of MZVs with rational coefficients given by
  \begin{equation*}
   \sum_{p,q,r\geq1}\frac{1}{(p+q+r)^nq^mr^l} = \sum_{i=0}^{l-1} \binom{m-1+i}{i} \zeta_\stuffle (n, m+i, l-i) +  \sum_{j=0}^{m-1} \binom{l-1+j}{j} \zeta_\stuffle (n, l+j, m-j).
  \end{equation*}

\end{prop}
\begin{proof}
 First observe that the cone underlying this conical sum is again $C_1$. So it is a linear combination of MZVs with rational coefficients from the same argument than the one of Example \ref{ex:simple_cone}. To compute them we need to apply the branched binarisation map $\fraks^T$ and the flattening map $fl_0$ to 
 \begin{equation*}
  \overline\Psi(C_1,(n,m,l)) = \tdtroisun{$n$}{$~l$}{$m$}.
 \end{equation*}
 For the computation, from Proposition \ref{prop:forest_cones} and Corollary \ref{coro:TZVs_stuffle_MZVs} we deduce
$$ \zeta(C_1,(n,m,l) ) = \zeta_\stuffle \circ fl_\yew \left( \tdtroisun{$n$}{$~l$}{$m$} \right). $$
Using Theorem \ref{thm:yew_formula} we obtain
$$ \tdun{$m$} \yew \tdun{$l$} = \sum_{i=0}^{l-1} \binom{m-1+i}{i} \tddeux{m+i}{l-i} 
+  \sum_{j=0}^{m-1} \binom{l-1+j}{j}\tddeux{l+j}{m-j}
$$
The result then follows from the definition of $fl_\yew$ (Definition \ref{defn:flattening_yew}).

\end{proof}
\begin{rk}
 The two computations above are a good illustration of an important point concerning the number-theoretic content of TZVs. It depends only of the weight of $||\vec s||:=s_1+\cdots+s_n$. Providing the standard conjectures on MZVs hold, $\zeta^t(\overline\Psi(C,\vec s))$ is always a linear combination of MZVs of weight exactly $||\vec s||$ with rational coefficients, for any decorated cone $(C,\vec s)\in\mathcal{CT}_{N^*}$ such that $\overline\Psi(C,\vec s)$ is a convergent forest. This is not the case for general CZVs, where some cancellations might happens and some terms in the expression of the CZV might be of lower weight than expected (or, of course, not be rational combination of MZVs). I am very thankful to E. Panzer who pointed out this fact to me and gave me an example.
\end{rk}
The two previous computations were special cases of Mordell-Tornheim zetas. However, the approach exposed above allows to compute many more conical sums. We present two more computations without every intermediate steps.
\begin{ex} \label{ex:less_simple}
 \begin{equation*}
  \zeta(C_2,\vec s_2) := \sum_{p,q,r,s,t\geq1}\frac{1}{(p+q+r+s+t)^4(q+t)^2rst}
 \end{equation*}
 We have 
 \begin{equation*}
  A_{C_2} = \begin{pmatrix}
             1 & 1 & 1 & 1 & 1 \\
             0 & 1 & 0 & 0 & 1 \\
             0 & 0 & 1 & 0 & 0 \\
             0 & 0 & 0 & 1 & 0 \\
             0 & 0 & 0 & 0 & 1 
            \end{pmatrix}, \quad 
  B_{C_2} = \begin{pmatrix}
             0 & 1 & 1 & 1 & 0 \\
             0 & 0 & 0 & 0 & 1 \\
             0 & 0 & 0 & 0 & 0 \\
             0 & 0 & 0 & 0 & 0 \\
             0 & 0 & 0 & 0 & 0
            \end{pmatrix}.
 \end{equation*}
 One easily checks that $C$ satisfies the hypothesis of Proposition \ref{prop:cone_to_tree}. Applying the algorithm above, one readily finds
 \begin{equation*}
  \sum_{p,q,r,s,t\geq1}\frac{1}{(p+q+r+s+t)^4(q+t)^2rst} = 2\zeta_\stuffle(4,1,1,2,1) + 6\zeta_\stuffle(4,1,2,1,1) + 12\zeta_\stuffle(4,2,1,1,1).
 \end{equation*}
\end{ex}
\begin{rk}
 Examples \ref{ex:simple_cone} and \ref{ex:less_simple} were kindly checked by E. Panzer using his HyperInt package. This package uses an integral representation of CZVs that differs from the one of Theorem \ref{thm:integral_sum}.
\end{rk}

The following last example required more complicated computations that we will also not detail.
\begin{ex}
 \begin{equation*}
  \zeta(C_3,\vec s_3) := \sum_{n_1,\cdots,n_7\geq1}\frac{1}{(n_1+\cdots+n_7)^5(n_2+\cdots+n_7)^2n_3(n_4)^2(n_5+n_6+n_7)^2n_6n_7}.
 \end{equation*}
 We have
 \begin{equation*}
  A_{C_3}=\begin{pmatrix}
           1 & 1 & 1 & 1 & 1 & 1 & 1 \\
           0 & 1 & 1 & 1 & 1 & 1 & 1 \\
           0 & 0 & 1 & 0 & 0 & 0 & 0 \\
           0 & 0 & 0 & 1 & 0 & 0 & 0 \\
           0 & 0 & 0 & 0 & 1 & 1 & 1 \\
           0 & 0 & 0 & 0 & 0 & 1 & 0 \\
           0 & 0 & 0 & 0 & 0 & 0 & 1
          \end{pmatrix},\quad 
  B_{C_3}=\begin{pmatrix}
           0 & 1 & 0 & 0 & 0 & 0 & 0 \\
           0 & 0 & 1 & 1 & 1 & 0 & 0 \\
           0 & 0 & 0 & 0 & 0 & 0 & 0 \\
           0 & 0 & 0 & 0 & 0 & 0 & 0 \\
           0 & 0 & 0 & 0 & 0 & 1 & 1 \\
           0 & 0 & 0 & 0 & 0 & 0 & 0 \\
           0 & 0 & 0 & 0 & 0 & 0 & 0 
          \end{pmatrix}
 \end{equation*}
 and the hypothesis of Proposition \ref{prop:cone_to_tree} are satisfied. After some more computations we obtain
 \begin{align*}
  & \sum_{n_1,\cdots,n_7\geq1}\frac{1}{(n_1+\cdots+n_7)^5(n_2+\cdots+n_7)^2n_3(n_4)^2(n_5+n_6+n_7)^2n_6n_7} \\
  =~ & 8\zeta_\stuffle(5,2,1,2,2,1,1) + 16\zeta_\stuffle(5,2,1,3,1,1,1) + 2\zeta_\stuffle(5,2,1,2,1,1,2) + 4\zeta_\stuffle(5,2,1,2,1,2,1) \\
  +~& 48\zeta_\stuffle(5,2,2,2,1,1,1) 
  + 28\zeta_\stuffle(5,2,2,1,2,1,1) + 8\zeta_\stuffle(5,2,2,1,1,1,2)   + 16\zeta_\stuffle(5,2,2,1,1,2,1) \\
  +~& 40\zeta_\stuffle(5,2,3,1,1,1,1) .
 \end{align*}
This result was kindly checked by Masataka Ono using their finite MZVs approach of \cite{OSY21}.

\end{ex}

\section*{Acknownledgements}

\addcontentsline{toc}{section}{Acknownledgements}

This paper stems from a discussion with Federico Zerbini, who suggested to me to look at conical zeta values and directed me toward some literature, including private communications. I am also very thanksful to Diego Andres Lopez Valencia, Lo\"ic Foissy and Erik Panzer, who read and commented a first version of this paper. They pointed out typos, mistakes and greatly improved its overall quality. I am also very thanksful to Masataka Ono for signaling to me the results available in the literature of finite MZVs and their generalisations and how they relate to this work and kindly answering my questions on this topic. Erik Panzer's comments about CZVs and the Dupont-Panzer-Zerbini conjecture were also very enlightening and I thank him for them.

 \addcontentsline{toc}{section}{References}
\printbibliography

@book{Wi96,
	author = {R. J. Wilson},
	title = {Introduction to graph theory},
	edition = {fourth},
	publisher = {Longman},
	year = {1996},
}

@article{Cl20,
	author = {P. J. Clavier},
	title = {Double shuffle relations for arborified zeta values},
	journal = {J. of Alg.}, 
	volume = {543}, 
	pages = {111-155},
	year = {2020},
	eprint = {arXiv:1812.00777v2},
}

@article{CGPZ18,
	author = {P. J. Clavier and  L. Guo and S. Paycha and B. Zhang},
	title = {Renormalisation and locality: Branched Zeta Values},
	eprint = {arXiv:1807.07630 [math-ph]},
	journal = {IRMA Lect. in Math. and Theor. Phys.},
	volume = {Vol. 32},
	year = {2020},
	pages = {85-132},
}

@book{Ecalle,
	author = {J. \'Ecalle},
	title = {Les fonctions r\'esurgentes I, II \& III},
	publisher = {Pr\'epublications math\'ematiques d'Orsay}, 
	year = {1981, 1982, 1985},
}

@article{Ma13,
	author = {D. Manchon},
	title = {Arborified multiple zeta values},
	journal = {Proceedings of "New approaches to Multiple Zeta Values"},
	publisher = {ICMAT, Madrid}, 
	year = {2013},
	eprint = {arXiv: 1603.01498 [math.CO]},
}

@article{GPZ13,
	author = {L. Guo and S. Paycha and B. Zhang},
	title = {Conical zeta values and their double subdivision relations},
	journal = {Advances in Mathematics},
	volume = {Vol. 252},
	pages = {343-381}, 
	DOI = {10.1016}, 
	eprint = {arXiv:1301.3370 [math.NT]},
}

@article{Te04,
	author = {T. Terasoma},
	title = {Rational convex cones and cyclotomic multiple zeta values},
	eprint = {arXiv:math/0410306 [math.AG]}, 
	year = {2004},
}

@article{Ze17,
	author = {F. Zerbini},
	title = {Elliptic multiple zeta values, modular graph functions and genus 1 superstring scattering amplitudes},
	year = {2017},
	eprint = {arXiv:1804.07989 [math.ph]},
}

@book{Fulton93,
	author = {W. Fulton}, 
	title = {Introduction to Toric Varieties}, 
	publisher = {Princeton University Press}, 
	year = {1993},
}

@book{Ziegler94,
	author = {G. Ziegler}, 
	title = {Lectures on Polytopes}, 
	series = {Grad. Texts in Math.}, 
	volume = {Vol. 152},
	publisher = {pringer-Verlag}, 
	year = {1994},
}

@article{To50,
	author = {L. Tornheim},
	title = {Harmonic double series},
	journal = {American J. Math.},
	volume = {Vol. 72},
	year = {1950},
	pages = {303-314},
}

@article{Mo58,
	author = {L. J. Mordell},
	title = {On the evaluation of some multiple series},
	journal = {J. London Math. Soc.},
	volume = {Vol. 33},
	year = {1958},
	pages = {271-368},
}

@article{Ho92,
	author = {M. E. Hoffman},
	title = {Multiple harmonic series},
	journal = {Pacific J. Math.},
	volume = {Vol. 152},
	year = {1992},
	pages = {275-290},
}

@article{Ts05,
	author = {H. Tsumura},
	title = {{On Mordell-Tornheim zeta values}},
	journal = {Proc. Amer. Math. Soc.},
	volume = {Vol. 133},
	year = {2005},
	pages = {2387-2393},
}

@article{BrZh10,
	author = {D. M. Bradley and X. Zhou},
	title = {On Mordell-Tornheim sums and multiple zeta values},
	journal = {Ann. Sci. Math. Québec},
	volume = {Vol. 34},
	year = {2010},
	pages = {15-23},
}

@book{Wa11,
	author ={M. Waldschmidt},
	title = {Lectures on Multiple Zeta Values}, 
	Publisher = {IMSc},
	year = {2011},
}

@article{To17,
	author = {I. Todorov},
	title = {Number theory meets high energy physics},
	journal = {Phys. of Part. and Nucl. Lett.},
	volume = {Vol. 14}, 
	pages = {291–297},
	year = {2017},
}

@article{Ya14,
	author = {S. Yamamoto},
	title = {Multiple zeta-star values and multiple integrals},
 Journal = {RIMS K{\^o}ky{\^u}roku Bessatsu},
 ISSN = {1881-6193},
 Volume = {B68},
 Pages = {3--14},
 Year = {2017},
	eprint = {arXiv: 1405.6499 [math.NT]},
}

@Article{Ya20,
 Author = {S. Yamamoto},
 Title = {Integrals associated with 2-posets and applications to multiple zeta values},
 Journal = {RIMS K{\^o}ky{\^u}roku Bessatsu},
 ISSN = {1881-6193},
 Volume = {B83},
 Pages = {27--46},
 Year = {2020},
 Language = {English},
}

@InProceedings{GPZ17,
author="L. Guo and S. Paycha and B. Zhang",
editor="F. Fauvet and D. Manchon and S. Marmi and D. Sauzin",
title="Renormalised conical zeta values",
booktitle="Resurgence, Physics and Numbers",
year="2017",
publisher="Scuola Normale Superiore",
pages="299--327",
isbn="978-88-7642-613-1",
}

@InProceedings{Ma03,
	author = {K. Matsumoto}, 
	title = {On Mordell-Tornheim and other multiple zeta functions},
	booktitle = {Proceedings of the Session in Analytic Number Theory and Diophantine Equations}, 
	editor = {D. R. Heath-Brown and B. Z. Moroz}, 
	publisher = {Bonner Math. Schriften 360, Bonn}, 
	number = {25},
	year = {2003},
}

@article{Lo21,
	author = {D. Lopez Valencia},
	title = {{Pole structure of Shintani zeta functions and Newton Polytopes}},
	journal = {In preparation},
	year = {2022},
}

@article{Ze16,
	author = {F. Zerbini},
	title = {Single-valued multiple zeta values in genus 1 superstring amplitudes},
	journal = {Comm. Number Th. Phys.},
	volume = {10},
	number = {4},
	year = {2016},
	pages = {703--737},
	eprint = {arXiv:1512.05689 [hep-th]},
}

@article{On16,
author = {M. Ono},
year = {2016},
title = {Finite multiple zeta values associated with 2-colored rooted trees},
volume = {181},
journal = {Journal of Num. Th.},
doi = {10.1016/j.jnt.2017.05.019}
}

@article{Ka16,
author = {K. Kamano},
title = {Finite {M}ordell-{T}ornheim multiple zeta values},
volume = {54},
journal = {Functiones et Approximatio Commentarii Mathematici},
number = {1},
pages = {65 -- 72},
year = {2016},
doi = {10.7169/facm/2016.54.1.6},
}

@article{BTT21,
author = {H. Bachmann and Y. Takeyama and K. Tasaka},
title = {Finite and symmetric Mordell–Tornheim multiple zeta values},
volume = {73},
journal = {J. of the Math. Soc. of Japan},
number = {4},
pages = {1129 -- 1158},
year = {2021},
doi = {10.2969/jmsj/84348434},
}

@article{OSY21,
author = {M. Ono and S. Shin-ichiro and S. Yamamoto},
title = {Truncated t-adic symmetric multiple zeta values and double shuffle relations},
year = {2021},
journal = {Res. in Num. Th.},
volume = {7},
series = {15},
number = {15},
doi = {10.1007/s40993-021-00241-5},
}

@book{Eu1796,
	author = {L. Euler},
	title = {Introductio in analysin infinitorum},
	publisher = {2 volums}, 
	year = {1796},
}

@article{Za91,
	author = {D. Zagier},
	title = {Values of Zeta functions and their applications},
	journal = {First Europeen Congress of Mathematics}, 
	volume = {II},
	pages = {497-512}, 
	publisher = {Birkh\"auser},
	year = {1994},
}


\end{document}